\documentclass[12pt]{amsart}
\usepackage[colorlinks=true, pdfstartview=FitV, linkcolor=blue, citecolor=blue, urlcolor=blue, breaklinks=true]{hyperref}
\usepackage{amsmath,amsfonts,amssymb,amsthm,amscd,comment,euscript,enumitem,booktabs}
\usepackage[usenames]{color}
\usepackage[all]{xy}

\newcommand\Z{\mathbb{Z}}

\newcommand\N{\mathbb{N}}
\newcommand\kk{{\Bbbk}}

\newcommand\id{\mathrm{id}}
\newcommand\rat{\mathrm{rat}}
\newcommand\ev{\mathrm{ev}}

\newcommand\rss{\mathrm{ss}}
\newcommand\ab{\mathrm{ab}}

\newcommand\g{\mathfrak{g}}
\newcommand\h{\mathfrak{h}}
\newcommand\n{\mathfrak{n}}

\newcommand\fb{\mathfrak{b}}
\newcommand\fsl{\mathfrak{sl}}
\newcommand\fl{\mathfrak{l}}

\newcommand\sm{\mathsf{m}}
\newcommand\sM{\mathsf{M}}

\newcommand\cB{\mathcal{B}}
\newcommand\cL{\mathcal{L}}
\newcommand\cR{\mathcal{R}}

\newcommand\cE{\mathcal{E}}

\newcommand{\ts}{\textstyle}

\DeclareMathOperator{\Hom}{Hom}

\DeclareMathOperator{\End}{End}

\DeclareMathOperator{\Spec}{Spec}
\DeclareMathOperator{\maxSpec}{maxSpec}

\DeclareMathOperator{\Ann}{Ann}
\DeclareMathOperator{\Supp}{Supp} 

\DeclareMathOperator{\rad}{rad}
\DeclareMathOperator{\hgt}{ht}

\theoremstyle{plain}
\newtheorem{theo}{Theorem}[section]
\newtheorem*{theo*}{Theorem}
\newtheorem{prop}[theo]{Proposition}
\newtheorem{lem}[theo]{Lemma}
\newtheorem{cor}[theo]{Corollary}

\theoremstyle{definition}
\newtheorem{defin}[theo]{Definition}
\newtheorem*{rem*}{Remark}
\newtheorem{rem}[theo]{Remark}

\newtheorem{example}[theo]{Example}

\numberwithin{equation}{section}

\allowdisplaybreaks

\begin{document}
%

\title{Equivariant map superalgebras}

\author{Alistair Savage}
\address{Department of Mathematics and Statistics, University of Ottawa}
\email{alistair.savage@uottawa.ca}
\thanks{This work was supported by a Discovery Grant from the Natural Sciences and Engineering Research Council of Canada.}

\subjclass[2010]{17B65, 17B10}


\keywords{Lie superalgebra, basic classical Lie superalgebra, loop superalgebra, equivariant map superalgebra, finite dimensional representation, finite dimensional module}

\begin{abstract}
  Suppose a group $\Gamma$ acts on a scheme $X$ and a Lie superalgebra $\g$.  The corresponding \emph{equivariant map superalgebra} is the Lie superalgebra of equivariant regular maps from $X$ to $\g$.  We classify the irreducible finite dimensional modules for these superalgebras under the assumptions that the coordinate ring of $X$ is finitely generated, $\Gamma$ is finite abelian and acts freely on the rational points of $X$, and $\g$ is a basic classical Lie superalgebra (or $\mathfrak{sl}(n,n)$, $n \ge 1$, if $\Gamma$ is trivial).  We show that they are all (tensor products of) generalized evaluation modules and are parameterized by a certain set of equivariant finitely supported maps defined on $X$.  Furthermore, in the case that the even part of $\g$ is semisimple, we show that all such modules are in fact (tensor products of) evaluation modules.  On the other hand, if the even part of $\g$ is not semisimple (more generally, if $\g$ is of type I), we introduce a natural generalization of Kac modules and show that all irreducible finite dimensional modules are quotients of these.  As a special case, our results give the first classification of the irreducible finite dimensional modules for twisted loop superalgebras.
\end{abstract}

\maketitle \thispagestyle{empty}

\tableofcontents

%
\section{Introduction}
%

Lie superalgebras, which are generalizations of Lie algebras to include a $\Z_2$-grading, are of considerable interest to both mathematicians and physicists.  Significant mathematical progress has been made in the theory of these algebras.  For example, Kac has classified the simple finite dimensional Lie superalgebras over an algebraically closed field of characteristic zero in \cite{Kac77} and the irreducible finite dimensional representations of the so-called basic classical Lie superalgebras in \cite{Kac77c,Kac77,Kac78}.  Nevertheless, the general theory of Lie superalgebras and their representations still remains much less developed than the corresponding theory of Lie algebras.  For example, while tremendous progress has been made in the classification of irreducible finite dimensional representations of (twisted) loop algebras and their generalizations (equivariant map algebras), the analogous theory in the super case is in its infancy.

In the current paper, we consider a certain important class of Lie superalgebras.  Let $X$ be a scheme and let $\g$ be a ``target'' Lie superalgebra, both defined over an algebraically closed field of characteristic zero. Furthermore, let $\Gamma$ be a finite group acting on $X$ and $\g$ by automorphisms.  Then the Lie superalgebra $(\g \otimes A)^\Gamma$ of $\Gamma$-equivariant regular maps from $X$ to $\g$ is called an \emph{equivariant map superalgebra}.  In the case that $\g$ is, in fact, a Lie algebra, we call it an \emph{equivariant map algebra}.

A special important case of the above construction is when $X$ is the one-dimensional torus.  In this case, $(\g \otimes A)^\Gamma$ is called a \emph{twisted loop (super)algebra} when $\Gamma$ acts nontrivially and an \emph{untwisted loop (super)algebra} in the case that $\Gamma$ acts trivially.  Loop (super)algebras and their twisted analogues play an vital role in the theory of affine Lie (super)algebras and quantum affine algebras.  They are also an important ingredient in string theory.  In the non-super case, their representation theory is fairly well developed.  In particular, the irreducible finite dimensional representations were classified by Chari and Pressley (\cite{Cha86,CP86,CP98}).  Subsequently, many generalizations of this work (also in the non-super case) have appeared in the literature, for example, in \cite{Bat04,CFK,CFS08,CM04,FL04,Lau10,Li04,Rao93,Rao01}.  In \cite{NSS09}, the irreducible finite dimensional representations were classified in the completely general setting of equivariant map algebras.  There it was shown that all such representations are tensor products of evaluation representations and one-dimensional representations.  In \cite{NS11}, the extensions between these representations were computed and the block decompositions of the categories of finite dimensional representations were described.

Despite the above-mentioned progress in our knowledge of the representation theory of equivariant map algebras, relatively little is known if the target $\g$ is in fact a Lie superalgebra (with nonzero odd part).  Perhaps the best starting point is when $\g$ is one of the \emph{basic classical Lie superalgebras} (see Section~\ref{subsec:Lie-superalg}), since these have many properties in common with semisimple Lie algebras.  In the special case of untwisted multiloop superalgebras with basic classical target, the irreducible finite dimensional representations have been classified in \cite{RZ04,Rao11}.  However, beyond this, almost nothing is known.  For example, even the irreducible finite dimensional representations of \emph{twisted} loop superalgebras have not been classified.  This is in sharp contrast to the non-super case.

In the current paper we give a complete classification of the irreducible finite dimensional modules for an arbitrary equivariant map superalgebra when the target $\g$ is a basic classical Lie superalgebra and the group $\Gamma$ is abelian and acts freely on the set of rational points of a scheme $X$ with finitely generated coordinate algebra.  In particular, our classification covers the case of the twisted loop superalgebras (with basic classical target).  Our main result (Theorem~\ref{thm:classification-equivariant}) is that all irreducible finite dimensional modules are generalized evaluation modules and that they can be naturally parameterized by a certain set of equivariant finitely supported maps defined on $X$.  In fact, even more can be said.  In the case that the even part of $\g$ is semisimple, all irreducible finite dimensional modules are evaluation modules, just as for equivariant map algebras.  However, this is not the case if the even part of $\g$ is not semisimple.  In this situation (more generally, when $\g$ is of \emph{type I}), we introduce a natural generalization of \emph{Kac modules}, which are certain modules induced from modules for the equivariant map algebra $(\g_{\bar 0} \otimes A)^\Gamma$, where $\g_{\bar 0}$ is the even part of $\g$ (and hence a Lie algebra).  We then show that all irreducible finite dimensional modules can be described as irreducible quotients of these Kac modules.  In the untwisted setting (i.e.\ when $\Gamma$ is trivial), this classification also applies when $\g = \mathfrak{sl}(n,n)$, $n \ge 1$.

A natural generalization of the category of finite dimensional modules is the category of quasifinite modules (that is, modules with finite dimensional weight spaces).  In Theorem~\ref{thm:hw-quasifinite-classification}, we give a characterization of the quasifinite irreducible highest weight $(\g \otimes A$)-modules when $\g$ is any basic classical Lie superalgebra or $\mathfrak{sl}(n,n)$, $n \ge 1$.   We do \emph{not} assume that $A$ is finitely generated there.  Our characterization is similar to \cite[Prop.~5.1]{Sav11}, which describes the quasifinite modules for the map Virasoro algebras.

This paper is organized as follows.  In Section~\ref{sec:prelim} we review some results on commutative algebras and Lie superalgebras that we will need.  We introduce the equivariant map superalgebras in Section~\ref{sec:EMSAs}.  In Section~\ref{sec:quasifinite-hw} we discuss their quasifinite and highest weight modules and give a characterization of the quasifinite modules.  We define generalized evaluation modules and prove some important facts about them in Section~\ref{sec:eval-rep}.  In Section~\ref{sec:kac-modules}, we discuss a generalization of Kac modules to the setting of equivariant map superalgebras.  Finally, we classify the irreducible finite dimensional modules in the untwisted setting in Section~\ref{sec:classification-untwisted} and in the twisted setting in Section~\ref{sec:classification-equivariant}.

\medskip

\paragraph{\textbf{Notation}} We let $\N$ be the set of nonnegative integers and $\N_+$ be the set of positive integers.  Throughout, $\kk$ is an algebraically closed field of characteristic zero and all Lie superalgebras, associative algebras, tensor products, etc.\ are defined over $\kk$ unless otherwise specified.  We let $A$ denote the coordinate ring of a scheme $X$.  Thus $A$ is a commutative associative unital $\kk$-algebra.   We let $X_\rat$ denote the set of $k$-rational points of $X$.  Recall that $\sm \in \maxSpec A$ is a \emph{$k$-rational point} of $X$ if its residue field $A/\sm$ is $\kk$.  Thus $X_\rat \subseteq \maxSpec A$ and we have equality if $A$ is finitely generated.  We assume that $A$ is finitely generated in Sections~\ref{sec:eval-rep}--\ref{sec:classification-equivariant}.  When we refer to the \emph{dimension} of $A$, we are speaking of its dimension as a vector space over $\kk$ (as opposed to referring to a geometric dimension).  Similarly, when we say that an ideal $I$ of $A$ is of \emph{finite codimension}, we mean that $A/I$ is finite dimensional as a vector space over $\kk$.  We use the term \emph{reductive Lie algebra} only for finite dimensional Lie algebras.

\medskip

\paragraph{\textbf{Acknowledgements}}  The author would like to thank Shun-Jen Cheng, Daniel Daigle, Dimitar Grantcharov, Erhard Neher and Hadi Salmasian for helpful discussions.

\medskip

\paragraph{\textbf{Change log}} This document contains some small corrections to typographical errors that appear in the published version.  It was decided that these errors are not significant enough to warrant publishing an erratum, but that it would helpful to correct the arXiv version.  The changes are listed below.
\begin{itemize}
  \item In Section~\ref{subsec:Lie-superalg}, the statement that $\fsl(n,n)$ is a 1-dimensional central extension of $A(n,n)$ was corrected to state that $\fsl(n,n)$ is a 1-dimensional central extension of $A(n-1,n-1)$.
  \item In Table~\ref{table}, the range for $n$ for $\fsl(n,n)$ was changed from $n \ge 1$ to $n \ge 2$.
\end{itemize}

%
\section{Preliminaries} \label{sec:prelim}
%

In this section, we review some mostly well-known results on commutative associative algebras and Lie superalgebras that will be used in the sequel.  For the results on Lie superalgebras (Section~\ref{subsec:Lie-superalg}), we refer the reader to \cite{FSS00,Kac77,Kac77b} for further details.

\subsection{Commutative algebras} \label{subsec:comm-alg}

The \emph{support} of an ideal $I$ of an algebra $A$ is
\[
  \Supp I := \{\sm \in \maxSpec A\ |\ I \subseteq \sm\} \cong \maxSpec A/I.
\]
Note that the support of an ideal is often defined to be the set of prime (rather than maximal) ideals containing it.  So our definition is more restrictive.

\begin{lem} \label{lem:ideal-codim-implies-support}
  If $I$ is an ideal of finite codimension in an algebra $A$, then $I$ has finite support.
\end{lem}

\begin{proof}
  Let $I$ be an ideal of finite codimension in $A$.  It suffices to show that $A/I$ has a finite number of maximal ideals.  Let $\sm_1,\sm_2,\dots$ be a sequence of pairwise distinct maximal ideals of $A/I$.  Define $I_n = \sm_1 \cdots \sm_n$.  Then $I_{n+1} \subsetneq I_n$ and thus the length of the sequence $I_1, I_2, \dots$ is bounded by the dimension of $A/I$.
\end{proof}

\begin{lem} \label{lem:ideal-support-implies-codim}
  If $I$ is an ideal of finite support in a finitely generated algebra $A$, then $I$ is of finite codimension in $A$.
\end{lem}

\begin{proof}
  If $I$ is of finite support, then $A/I$ has finitely many maximal ideals.  Since $A$ is finitely generated, this implies that $A/I$ is finite dimensional (see, for example, \cite[Lem.~1.9.2]{Fu11}).
\end{proof}

\begin{lem} \label{lem:ideal-product-intersection}
  If $I$ and $J$ are ideals of an algebra $A$ and $I$ and $J$ have disjoint supports, then $IJ = I \cap J$.
\end{lem}

\begin{proof}
  Suppose $I$ and $J$ are ideals of an algebra $A$ and $I$ and $J$ have disjoint supports.  Then $I + J = A$ since $(\Supp I) \cap (\Supp J) = \Supp (I+J)$.  Thus we can write $1=i+j$ with $i \in I$ and $j \in J$.  Let $x \in I \cap J$.  Then $x = xi + xj$, where $xi,xj \in IJ$.  Therefore $I \cap J \subseteq IJ$.  The reverse inclusion is obvious.
\end{proof}

\begin{lem}[{\cite[Prop.~7.14]{AM69}}] \label{lem:ideal-contains-power-of-radical}
  In a Noetherian ring, every ideal contains a power of its radical.
\end{lem}

\subsection{Lie superalgebras} \label{subsec:Lie-superalg}

Let $\g = \g_{\bar 0} \oplus \g_{\bar 1}$ be a finite dimensional simple Lie superalgebra over $\kk$.  Then $\g$ is called \emph{classical} if the representation of $\g_{\bar 0}$ on $\g_{\bar 1}$ is completely reducible.  A simple Lie superalgebra is classical if and only if its even part $\g_{\bar 0}$ is a reductive Lie algebra.  If $\g$ is a classical Lie superalgebra, then the representation of $\g_{\bar 0}$ on $\g_{\bar 1}$ is either
\begin{enumerate}
  \item irreducible, in which case we say $\g$ is of \emph{type II}, or
  \item the direct sum of two irreducible representations, in which case we say $\g$ is of \emph{type I}.
\end{enumerate}
A classical Lie superalgebra $\g$ is called \emph{basic} if there exists a non-degenerate invariant bilinear form on $\g$.  (The classical Lie superalgebras that are not basic are called \emph{strange}.)

In Table~\ref{table}, we list all of the basic classical Lie superalgebras (up to isomorphism) that are not Lie algebras, together with their even part and their type.  We also include the Lie superalgebra $\fsl(n,n)$, $n \ge 2$, which is a 1-dimensional central extension of $A(n-1,n-1)$.

\begin{table}[!h]
  \begin{tabular}{ccc}
    \toprule
    $\g$ & $\g_{\bar 0}$ & Type \\
    \midrule
    $A(m,n)$, $m > n \ge 0$ & $A_m \oplus A_n \oplus \kk$ & I \\
    $A(n,n)$, $n \ge 1$ & $A_n \oplus A_n$ & I \\
    $\fsl(n,n)$, $n \ge 2$ & $A_{n-1} \oplus A_{n-1} \oplus \kk$ & N/A \\
    $C(n+1)$, $n \ge 1$ & $C_n \oplus \kk$ & I \\
    $B(m,n)$, $m \ge 0$, $n \ge 1$ & $B_m \oplus C_n$ & II \\
    $D(m,n)$, $m \ge 2$, $n \ge 1$ & $D_m \oplus C_n$ & II \\
    $F(4)$ & $A_1 \oplus B_3$ & II \\
    $G(3)$ & $A_1 \oplus G_2$ & II \\
    $D(2,1;\alpha)$, $\alpha \ne 0,-1$ & $A_1 \oplus A_1 \oplus A_1$ & II \\
    \bottomrule
  \end{tabular}
  \bigskip
  \caption{The basic classical Lie superalgebras that are not Lie algebras, together with their even part and their type} \label{table}
\end{table}

Note that in all cases in Table~\ref{table}, $\g_{\bar 0}$ is either semisimple or reductive with one-dimensional center.  Also note that all the Lie superalgebras in Table~\ref{table}, including $\mathfrak{sl}(n,n)$, are perfect (i.e.\ satisfy $[\g,\g]=\g$).

For any basic classical Lie superalgebra $\g$, there exists a \emph{distinguished $\Z$-grading} $\g = \bigoplus_{i \in \Z} \g_i$ of $\g$ that is compatible with the $\Z_2$-grading (i.e.\ each graded piece is a $\Z_2$-graded subspace of $\g$) and such that
\begin{enumerate}
  \item \label{item:typeI-Zgrading} if $\g$ is of type I, then $\g_i=0$ for $|i|>1$, $\g_{\bar 0} = \g_0$, $\g_{\bar 1} = \g_{-1} \oplus \g_1$, and
  \item if $\g$ is of type II, then $\g_i=0$ for $|i|>2$, $\g_{\bar 0} = \g_{-2} \oplus \g_0 \oplus \g_2$, $\g_{\bar 1} = \g_{-1} \oplus \g_1$.
\end{enumerate}
If $\g = \mathfrak{sl}(n,n)$, then we have a distinguished $\Z$-grading as in \eqref{item:typeI-Zgrading} above.  We simply let $\g_0$ be the preimage of the zero graded piece of the type I Lie superalgebra $A(n,n)$ under the canonical projection from $\g$ to $A(n,n)$.

By definition, a Cartan subalgebra of $\g$ is just a Cartan subalgebra of the even part $\g_{\bar 0}$.  Fix such a Cartan subalgebra $\h$.  If $\g$ is a basic classical Lie superalgebra or $\mathfrak{sl}(n,n)$, then we can choose a Borel subalgebra $\fb^+$ containing $\h$ and $\bigoplus_{i > 0} \g_i$.  (Here we refer to the distinguished $\Z$-grading on $\g$ mentioned above.)  Since the adjoint action of $\h$ on $\g$ is diagonalizable, we obtain a decomposition
\[
  \g = \n^- \oplus \h \oplus \n^+,
\]
where $\n^{\pm}$ are subalgebras, $[\h,\n^{\pm}] \subseteq \n^{\pm}$, and $\fb^+ = \h \oplus \n^+$.  A root $\alpha$ is called \emph{positive} (resp.\ \emph{negative}) if $\g_\alpha \cap \n^+ \ne 0$ (resp.\ $\g_\alpha \cap \n^- \ne 0$).  A root $\alpha$ is called \emph{even} (resp.\ \emph{odd}) if $\g_\alpha \cap \g_{\bar 0} \ne 0$ (resp.\ $\g_\alpha \cap \g_{\bar 1} \ne 0$).  Let $\Delta_{\bar 0}$ and $\Delta_{\bar 1}$ denote the set of even and odd roots respectively.  Then $\Delta = \Delta_{\bar 0} \cup \Delta_{\bar 1}$ is the set of all roots.  We let $\Delta^\pm_{\bar 0}$, $\Delta_{\bar 1}^\pm$, $\Delta^\pm$ denote the corresponding subsets of positive/negative roots.  For $\g$ a basic classical Lie superalgebra, the zero weight space of $\g$ is equal to $\h$ (see \cite[Prop.~5.3]{Kac77b}).   It follows that this property also holds for $\mathfrak{sl}(n,n)$.  Let $P$ and $Q^+$ be the weight lattice and positive root lattice of $\g$ respectively.

\begin{lem}[{\cite[Prop.~5.2.4]{Kac77}}] \label{lem:fd-reps-solvable}
  All the irreducible finite dimensional representations of a solvable Lie superalgebra $L = L_{\bar 0} \oplus L_{\bar 1}$ are one-dimensional if and only if $[L_{\bar 1}, L_{\bar 1}] \subseteq [L_{\bar 0},L_{\bar 0}]$.
\end{lem}

\begin{lem} \label{lem:annilating-ideal}
  Suppose $L$ is a Lie superalgebra and $V$ is an irreducible $L$-module such that $\mathfrak{I} v=0$ for some ideal $\mathfrak{I}$ of $L$ and nonzero vector $v \in V$.  Then $\mathfrak{I} V=0$.
\end{lem}

\begin{proof}
  Let
  \[
    W = \{w \in V\ |\ \mathfrak{I} w = 0\}.
  \]
  By assumption $v \in W$ and so $W$ is nonzero.  Furthermore, since $\mathfrak{I}$ is an ideal of $L$, it is easy to see that $W$ is a submodule of $V$.  Since $V$ is irreducible, this implies that $W=V$, completing the proof of the lemma.
\end{proof}

In the Lie superalgebra case, the Poincar\'e-Birkhoff-Witt Theorem (or PBW Theorem) has the following form (see, for example, \cite[Thm.~7.2.1]{Var04}).

\begin{lem} \label{lem:PBW}
  Let $B_0, B_1$ be totally ordered bases for $\g_{\bar 0}, \g_{\bar 1}$ (over $\kk$) respectively.  Then the standard monomials
  \[
    u_1 \cdots u_r v_1 \cdots v_s, \quad u_1,\dots,u_r \in B_0,\ v_1, \dots, v_s \in B_1,\quad u_1 \le \cdots \le u_r,\ v_1 < \cdots < v_s,
  \]
   form a basis (over $\kk$) for the universal enveloping superalgebra $U(\g)$. In particular, if $\g_{\bar 0} = 0$ (i.e.\ $\g$ is odd), then $U(\g)$ is finite dimensional.
\end{lem}

%
\section{Equivariant map superalgebras} \label{sec:EMSAs}
%

In this section we introduce our main object of study: the equivariant map superalgebras.  Recall that $\g = \g_{\bar 0} \oplus \g_{\bar 1}$ is a Lie superalgebra and $X$ is a scheme with coordinate ring $A$.

\begin{defin}[Map superalgebra] \label{def:map-superalgebra}
  We call the Lie superalgebra $\g \otimes A$ of regular functions on $X$ with values in $\g$ a \emph{map (Lie) superalgebra}.  The $\Z_2$-grading on $\g \otimes A$ is given by $(\g \otimes A)_\epsilon = \g_\epsilon \otimes A$ for $\epsilon = \bar 0, \bar 1$, and the multiplication on $\g \otimes A$ is pointwise.  That is, multiplication is given by extending the bracket
  \[
    [u_1 \otimes f_1, u_2 \otimes f_2] = [u_1,u_2] \otimes f_1f_2,\quad u_1,u_2 \in \g,\ f_1,f_2 \in A,
  \]
  by linearity.  It is easily verified that $\g \otimes A$ satisfies the axioms of a Lie superalgebra.
\end{defin}

An action of a group $\Gamma$ on a Lie superalgebra $\g$ and on a scheme $X$ will always be assumed to be by Lie superalgebra automorphisms of $\g$ and scheme automorphisms of $X$.  Recall that Lie superalgebra automorphisms respect the $\Z_2$-grading.  A $\Gamma$-action on $X$ induces a $\Gamma$-action on $A$.

\begin{defin}[Equivariant map superalgebra]
  Let $\Gamma$ be a group acting on a scheme $X$ (hence on $A$) and a Lie superalgebra $\g$ by automorphisms.  Then $\Gamma$ acts naturally on $\g \otimes A$ by extending the map $\gamma (u \otimes f) = (\gamma  u) \otimes (\gamma  f)$, $\gamma \in \Gamma$, $u \in \g$, $f \in A$, by linearity.  We define
  \[
    (\g \otimes A)^\Gamma = \{\mu \in \g \otimes A\ |\ \gamma \mu = \mu \ \forall\ \gamma \in \Gamma\}
  \]
  to be the subsuperalgebra of points fixed under this action.  In other words, $(\g \otimes A)^\Gamma$ is the subsuperalgebra of $\g \otimes A$ consisting of $\Gamma$-equivariant maps from $X$ to $\g$.  We call $(\g \otimes A)^\Gamma$ an \emph{equivariant map (Lie) superalgebra}.  Note that if $\Gamma$ is the trivial group, this definition reduces to Definition~\ref{def:map-superalgebra}.
\end{defin}

\begin{example}[Multiloop superalgebras] \label{eg:multiloop}
  Fix positive integers $n, m_1, \dots, m_n$.  Let
  \[
    \Gamma = \langle \gamma_1,\dots, \gamma_n\ |\ \gamma_i^{m_i}=1,\ \gamma_i \gamma_j = \gamma_j \gamma_i,\ \forall\ 1 \le i,j \le n \rangle \cong \Z/m_1 \Z \times \dots \times \Z/m_n \Z,
  \]
  and suppose that $\Gamma$ acts on $\g$. Note that this is equivalent to specifying commuting automorphisms $\sigma_i$, $i=1,\dots,n$, of $\g$ such that $\sigma_i^{m_i}=\id$. For $i = 1,\dots, n$, let $\xi_i$ be a primitive $m_i$-th root of unity. Let $X = \Spec A$, where $A = \kk[t_1^{\pm 1}, \dots, t_n^{\pm 1}]$ is the $\kk$-algebra of Laurent polynomials in $n$ variables (in other words, $X$ is the $n$-dimensional torus $(\kk^\times)^n$), and define an action of $\Gamma$ on $X$ by
  \[
    \gamma_i (z_1, \dots, z_n) = (z_1, \dots, z_{i-1}, \xi_i z_i, z_{i+1}, \dots, z_n).
  \]
  Then
  \begin{equation} \label{eq:twisted-multiloop-def}
      M(\g,\sigma_1,\dots,\sigma_n,m_1,\dots,m_n) := (\g \otimes A)^\Gamma
  \end{equation}
  is the \emph{(twisted) multiloop superalgebra} of $\g$ relative to $(\sigma_1, \dots, \sigma_n)$ and $(m_1, \ldots, m_n)$.  In the case that $\Gamma$ is trivial (i.e.\ $m_i=1$ for all $i=1,\dots,n$), we call often call it an \emph{untwisted multiloop superalgebra}.  If $n=1$, $M(\g,\sigma_1,m_1)$ is simply called a (\emph{twisted} or \emph{untwisted}) \emph{loop superalgebra}.
\end{example}

\begin{rem} \label{rem:reduction-to-affine}
  If $\Gamma$ acts on $X$, hence on $A$, then $\Gamma$ acts on $V := \Spec A$ by \cite[I-40]{EH00}.  Since the coordinate rings of $X$ and $V$ are both $A$, the equivariant map superalgebras corresponding to $X$ and $V$ are the same. Therefore, we lose no generality in assuming that $X$ is an affine scheme and we will often do so in the sequel.
\end{rem}

\begin{defin}[$\Ann_A V$, $\Ann_A \rho$] \label{def:annihilator}
  For a $(\g \otimes A)^\Gamma$-module $V$, we define $\Ann_A V$ to be the largest $\Gamma$-invariant ideal $I$ of $A$ satisfying $(\g \otimes I)^\Gamma V=0$. In other words, $\Ann_A V$ is the sum of all $\Gamma$-invariant ideals $I$ such that $(\g \otimes I)^\Gamma V =0$.  If $\rho$ is the representation corresponding to $V$, we set $\Ann_A \rho := \Ann_A V$.
\end{defin}

\begin{lem} \label{lem:g-perfect-annihilator}
  If $\g$ is a perfect Lie superalgebra (i.e.\ $[\g,\g] = \g$) and $V$ is a $(\g \otimes A$)-module, then
  \[
    \Ann_A V = \{f \in A\ |\ (\g \otimes f)V=0\}.
  \]
\end{lem}

\begin{proof}
  Let $I = \{f \in A\ |\ (\g \otimes f)V=0\}$.  It suffices to show that $I$ is an ideal of $A$.  It is easy to see that $I$ is a linear subspace of $A$.  Now let $f \in I$ and $g \in A$.  Then, since $\g$ is perfect,
  \[
    (\g \otimes fg)V = [\g \otimes f,\g \otimes g] V = 0,
  \]
  and so $fg \in I$. Thus $I$ is an ideal of $A$.
\end{proof}

\begin{defin}[Support] \label{def:support}
  Let $V$ be a $(\g \otimes A)^\Gamma$-module.  We define the \emph{support} of $V$ to be
  \[
    \Supp_A V := \Supp \Ann_A V.
  \]
  If $\rho$ is the representation corresponding to $V$, we also set $\Supp_A \rho := \Supp_A V$.  We say that $V$ has \emph{reduced support} if $\Ann_A V$ is a radical ideal (equivalently, $\Spec (A/\Ann_A V)$ is reduced).
\end{defin}

%
\section{Quasifinite and highest weight modules} \label{sec:quasifinite-hw}
%

In this section we introduce certain classes of modules that will play an important role in our exposition.  We assume that $\g = \g_{\bar 0} \oplus \g_{\bar 1}$ is either a reductive Lie algebra (considered as a Lie superalgebra with zero odd part), a basic classical Lie superalgebra, or $\mathfrak{sl}(n,n)$, $n \ge 1$.  Thus, in particular, $\g_{\bar 0}$ is a reductive Lie algebra.  If $\g$ is a basic Lie superalgebra or $\mathfrak{sl}(n,n)$, $n \ge 1$, we choose a triangular decomposition of $\g$ as in Section~\ref{subsec:Lie-superalg}.  If $\g$ is a reductive Lie algebra, we choose any triangular decomposition.  We set $\fb^\pm = \h \oplus \n^\pm$.  We also identify $\g$ with the subsuperalgebra $\g \otimes \kk \subseteq \g \otimes A$.

\begin{defin}[Weight module]
  A $(\g \otimes A)$-module $V$ is called a \emph{weight module} if its restriction to $\g$ is a weight module, that is, if
  \[ \ts
    V = \bigoplus_{\lambda \in \h^*} V_\lambda,\quad V_\lambda:= \{v \in V\ |\ h v = \lambda(h)v \ \forall\ h \in \h\}.
  \]
  The $\lambda \in \h^*$ such that $V_\lambda \ne 0$ are called \emph{weights} of $V$.  A nonzero element of $V_\lambda$ for some $\lambda \in \h^*$ is called a \emph{weight vector} of weight $\lambda$.
\end{defin}

\begin{defin}[Quasifinite module]
  A $(\g \otimes A)$-module is called \emph{quasifinite} if it is a weight module and all weight spaces are finite dimensional.
\end{defin}

\begin{defin}[Highest weight module]
  A $(\g \otimes A)$-module $V$ is called a \emph{highest weight module} if there exists a nonzero vector $v \in V$ such that $(\n^+ \otimes A)v = 0$, $U(\h \otimes A)v = \kk v$, and $U(\g \otimes A)v = V$.  Such a vector $v$ is called a \emph{highest weight vector}.
\end{defin}

\begin{rem}
  Note that every highest weight module is a weight module.  This follows from the fact that the highest weight vector is a weight vector by definition and generates the entire module.
\end{rem}

We fix the usual partial order on $P$ given by
\[
  \lambda \ge \mu \iff \lambda - \mu \in Q^+.
\]

\begin{lem} \label{lem:fd-implies-hw}
  Every irreducible finite dimensional $(\g \otimes A)$-module is a highest weight module.
\end{lem}

\begin{proof}
  Let $V$ be an irreducible finite dimensional $(\g \otimes A)$-module.  Since all irreducible finite dimensional representations of the abelian Lie algebra $\h \otimes A$ are one-dimensional, there exists a nonzero vector $w \in V$ fixed by $\h \otimes A$.  Thus $w$ is a weight vector.  Since $V$ is irreducible as a $(\g \otimes A)$-module, it follows that $w$ generates $V$ and hence that $V$ is a weight module.  Now let $\lambda$ be a maximal weight of $V$ and let $v$ be a nonzero vector in $V_\lambda$.  Then it follows from the irreducibility of $V$ that $v$ is a highest weight vector.
\end{proof}

\begin{defin}[Irreducible highest weight modules $V(\psi)$]
  Fix $\psi \in (\h \otimes A)^*$.  We define an action of $\fb^+ \otimes A$ on $\kk$ (considered to be in degree zero) by declaring $\h \otimes A$ to act via $\psi$ and $\n^+ \otimes A$ to act by zero.  We denote the resulting module by $\kk_\psi$ and consider the induced module
  \[
    U(\g \otimes A) \otimes_{U(\fb^+ \otimes A)} \kk_\psi.
  \]
  It is clear that this is a highest weight module.  It follows that it possesses a unique maximal submodule $N(\psi)$ and we define
  \[
    V(\psi) = \left(U(\g \otimes A) \otimes_{U(\fb^+ \otimes A)} \kk_\psi\right)/N(\psi).
  \]
  Every irreducible highest weight $(\g \otimes A$)-module (hence, by Lemma~\ref{lem:fd-implies-hw}, every irreducible finite dimensional $(\g \otimes A$)-module) is isomorphic to $V(\psi)$ for some $\psi \in (\g \otimes A)^*$.
\end{defin}

\begin{lem} \label{lem:Schur}
  Let $V$ be a highest weight $(\g \otimes A)$-module.  Then $\End_{\g \otimes A} V = \kk$.
\end{lem}

\begin{proof}
  This follows easily from the fact that any $(\g \otimes A)$-module endomorphism of $V$ must send a highest weight vector to a highest weight vector, together with the fact that the highest weight space of $V$ is one-dimensional.
\end{proof}

\begin{rem}
  Note that the algebra of endomorphisms of an irreducible module over a Lie superalgebra is \emph{not} always isomorphic to $\kk$.  Indeed, in some cases this algebra of endomorphisms is spanned (over $\kk$) by the identity and an involution interchanging the even and odd parts of the module (see \cite[p.~609]{Kac78}).
\end{rem}

\begin{prop}[Density Theorem] \label{prop:density-theorem}
  Suppose $V$ is an irreducible highest weight $(\g \otimes A)$-module, and let $\rho : U(\g \otimes A) \to \End_\kk V$ be the corresponding representation of the universal enveloping superalgebra.  Let $\kappa \in \End_\kk V$ and $v_1,\dots,v_n \in V$.  Then there exists $x \in U(\g \otimes A)$ such that $x v_i = \kappa(v_i)$ for $i=1,\dots,n$.  If $V$ is finite dimensional, then $\rho(U(\g \otimes A)) = \End_\kk V$.
\end{prop}

\begin{proof}
  By Lemma~\ref{lem:Schur}, we have $\End_{U(\g \otimes A)} V = \kk$.  The result then follows from the Jacobson Density Theorem (see, for example, \cite[Thm.~XVII.3.2]{Lan02}), where we consider only homogeneous elements.  See also \cite[Prop.~8.2]{Che95}.
\end{proof}

\begin{lem} \label{lem:outer-tensor-irred}
   Suppose $\g^1, \g^2$ are reductive Lie algebras, basic classical Lie superalgebras, or $\mathfrak{sl}(n,n)$, $n \ge 1$, and $A^1, A^2$ are commutative associative algebras.  If $V^i$ is an irreducible highest weight $(\g^i \otimes A^i)$-module for $i=1,2$, then $V^1 \otimes V^2$ is an irreducible highest weight module for $(\g^1 \otimes A^1) \oplus (\g^2 \otimes A^2)$.
\end{lem}

\begin{proof}
  For $i=1,2$, we have the triangular decomposition $\g^i = \n^{i,-} \oplus \h^i \oplus \n^{i,+}$ as described at the beginning of this section.  For each weight $\lambda$ of $V^2$, let $\cB_\lambda$ be a basis for $V^2_\lambda$.  Let $z$ be an arbitrary nonzero element of $V^1 \otimes V^2$.  Then $z$ can be written as
  \[
    z = \sum_\lambda \sum_{w \in \cB_\lambda} v_w \otimes w,
  \]
  where the first sum is over the weights of $V^2$ and the $v_w \in V^1$ are weight vectors with only a finite number of them nonzero.  Now, let $\nu$ be a minimal weight among the weights of the (nonzero) $v_w$ and fix a nonzero $v_{w'}$ of weight $\nu$.  Proposition~\ref{prop:density-theorem} and the PBW Theorem (Lemma~\ref{lem:PBW}) imply that there exists a weight vector $x_1 \in U(\n^{1,+} \otimes A^1)$ such that $x_1 v_{w'}$ is a highest weight vector of $V^1$.  Then it follows from our choices that, for all $w \in \cB_\lambda$, $\lambda$ a weight of $V^2$, the vector $x_1 v_w$ is a either zero or a highest weight vector of $V^1$.  Therefore, if $v^1$ is a highest weight vector of $V^1$, we have
  \[
    x_1z = v^1 \otimes \left( \sum_{\lambda \in S} \sum_{w \in \cB_\lambda} a_w w \right) \ne 0,
  \]
  for some $a_w \in \kk$, not all zero (but with only finitely many nonzero), and a nonempty finite set $S$ of weights of $V^2$.  An argument similar to the one above shows that there exists $x_2 \in U(\n^{2,+} \otimes A^2)$ such that
  \[
    x_2x_1z = v^1 \otimes v^2,
  \]
  where $v^2$ is a highest weight vector of $V^2$.  Since $v^1 \otimes v^2$ generates $V^1 \otimes V^2$ as a module for $(\g^1 \otimes A^1) \oplus (\g^2 \otimes A^2)$, so does $z$.  It follows that $V^1 \otimes V^2$ is irreducible.  It is also clearly highest weight, with highest weight vector $v^1 \otimes v^2$.
\end{proof}

\begin{cor} \label{cor:irreds-for-tensor-prods}
  Suppose $\g^1, \g^2$ are reductive Lie algebras, basic classical Lie superalgebras, or $\mathfrak{sl}(n,n)$, $n \ge 1$, and $A^1, A^2$ are commutative associative algebras.  Then any irreducible finite dimensional module $V$ for $(\g^1 \otimes A^1) \oplus (\g^2 \otimes A^2)$ is of the form $V^1 \otimes V^2$ for irreducible finite dimensional $(\g^i \otimes A^i)$-modules $V^i$, $i=1,2$.
\end{cor}

\begin{proof}
  Suppose $V$ is an irreducible finite dimensional module for $(\g^1 \otimes A^1) \oplus (\g^2 \otimes A^2)$.  Then it is also an irreducible module for $U((\g^1 \otimes A^1) \oplus (\g^2 \otimes A^2)) \cong U(\g^1 \otimes A^1) \otimes U(\g^2 \otimes A^2)$ (tensor product as superalgebras).  By \cite[Prop.~8.4]{Che95}, there exist irreducible $U(\g^i \otimes A^i)$-modules $V^i$, $i=1,2$, such that $V$ is either isomorphic to $V^1 \otimes V^2$ or a proper submodule of this tensor product.  Since this tensor product is irreducible by Lemma~\ref{lem:outer-tensor-irred}, the result follows.
\end{proof}

\begin{prop} \label{prop:tensor-product}
  The tensor product of two irreducible highest weight $(\g \otimes A)$-modules with disjoint supports is irreducible.
\end{prop}

\begin{proof}
  Suppose $V^1$ and $V^2$ are irreducible highest weight $(\g \otimes A)$-modules with disjoint supports.  For $i=1,2$, let $I_i = \Ann_A V^i$ and let $\rho_i : \g \otimes A \to \End V^i$ be the representation corresponding to $V^i$.  The representation $\rho_1 \otimes \rho_2$ factors through the composition
  \begin{equation} \label{eq:tensor-product-factor-map}
    \g \otimes A \stackrel{\Delta}{\hookrightarrow} (\g \otimes A) \oplus (\g \otimes A) \twoheadrightarrow (\g \otimes A/I_1) \oplus (\g \otimes A/I_2),
  \end{equation}
  where $\Delta$ is the diagonal embedding and the second map is the obvious projection on each summand.  Since the supports of $I_1$ and $I_2$ are disjoint, we have $I_1 \cap I_2 = I_1 I_2$ by Lemma~\ref{lem:ideal-product-intersection}.   Therefore $A/I_1I_2 \cong (A/I_1) \oplus (A/I_2)$.  We thus have the following commutative diagram:
  \[
    \xymatrix{
      & (\g \otimes A) \oplus (\g \otimes A) \ar@{->>}[dr] & \\
      \g \otimes A \ar[ur]^(.4){\Delta} \ar@{->>}[dr] & & (\g \otimes A/I_1) \oplus (\g \otimes A/I_2) \\
      & \g \otimes A/I_1I_2 \ar[ur]_{\cong} &
    }
  \]
  It follows that the composition~\eqref{eq:tensor-product-factor-map} is surjective.  Since $V^1 \otimes V^2$ is irreducible as a module for $(\g \otimes A/I_1) \oplus (\g \otimes A/I_2)$ by Lemma~\ref{lem:outer-tensor-irred}, the result follows.
\end{proof}

\begin{cor} \label{cor:psi-disjoint-support-product}
  Suppose that $\psi_1,\dots,\psi_n \in (\h \otimes A)^*$ such that $V(\psi_1), \dots, V(\psi_n)$ have disjoint supports.  Then
  \[ \ts
    V \left( \sum_{i=1}^n \psi_i \right) = \bigotimes_{i=1}^n V(\psi_i).
  \]
\end{cor}

\begin{proof}
  It suffices to prove the result for $n=2$, since the general result will then follow by induction.  Let $\lambda_i$ be the highest weight of $V(\psi_i)$, $i=1,2$.  Then $V(\psi_1)_{\lambda_1} \otimes V(\psi_2)_{\lambda_2}$ is a one-dimensional weight space of $V(\psi_1) \otimes V(\psi_2)$ of maximal weight $\lambda_1 + \lambda_2$.  Since $(\h \otimes A)^*$ clearly acts on this space via $\psi_1 + \psi_2$, it suffices to show that $V(\psi_1) \otimes V(\psi_2)$ is irreducible.  This follows immediately from Proposition~\ref{prop:tensor-product}.
\end{proof}

\begin{cor} \label{cor:prod-fd-irr-is-irr}
  The tensor product of irreducible finite dimensional $(\g \otimes A)$-modules with disjoint supports is irreducible.
\end{cor}

\begin{proof}
  This follows from Lemma~\ref{lem:fd-implies-hw} and Proposition~\ref{prop:tensor-product}.
\end{proof}

\begin{lem} \label{lem:psi-annilator}
  Suppose that $\psi \in (\h \otimes A)^*$ and that $I$ is an ideal of $A$.  Then $\psi(\h \otimes I) = 0$ if and only if $(\g \otimes I) V(\psi)=0$.
\end{lem}

\begin{proof}
  It is obvious that $(\g \otimes I)V(\psi)=0$ implies $\psi(\h \otimes I)=0$.  It thus remains to prove the other implication.

  Suppose $\psi \in (\h \otimes A)^*$ and $\psi(\h \otimes I)=0$ for some ideal $I$ of $A$.  Let $v$ be a highest weight vector of $V(\psi)$.  By Lemma~\ref{lem:annilating-ideal}, it suffices to show that $(\g \otimes I)v=0$.  We have $(\n^+ \otimes I)v=0$ since $v$ is a highest weight vector and $(\h \otimes I)v=0$ by the assumption on $I$.  For $\lambda = - \sum n_i \alpha_i$, $n_i \in \N$, where the $\alpha_i$ are the simple roots of $\g$, we define the \emph{height} of $\lambda$ to be $\hgt \lambda := \sum n_i$.  We show by induction on the height of $\lambda$ that $(\fb^-_\lambda \otimes I)v=0$ (here $\fb^-_\lambda$ is the weight space of $\fb^-$ of weight $\lambda$).  We already have the result for $\hgt \lambda = 0$ (since $\fb^-_0 = \h$).  Suppose that, for some $m \ge 0$, the results holds whenever $\hgt \lambda \le m$.  Fix $\lambda \in Q^- := - Q^+$ with $\hgt \lambda = m+1$.  Then
  \[
    (\n^+ \otimes A)(\fb^-_\lambda \otimes I)v \subseteq [\n^+ \otimes A, \fb^-_\lambda \otimes I]v + (\fb^-_\lambda \otimes I)(\n^+ \otimes A)v = ([\n^+,\fb^-_\lambda] \otimes I)v = 0,
  \]
  where the last equality follows from the induction hypothesis since any element of $[\n^+,\fb^-_\lambda]$ is either an element of $\fb^-_{\lambda'}$ with $\hgt \lambda' < \hgt \lambda$ or an element of $\n^+$.  This completes the proof of the inductive step.
\end{proof}

The following theorem classifies the highest weight quasifinite $(\g \otimes A)$-modules.

\begin{theo} \label{thm:hw-quasifinite-classification}
  Suppose $V=V(\psi)$ is an irreducible highest weight $(\g \otimes A$)-module, where $\g$ is either a basic classical Lie superalgebra or $\mathfrak{sl}(n,n)$, $n \ge 1$.  Then the following conditions are equivalent:
  \begin{enumerate}
    \item \label{item:V-quasifinite} The module $V$ is quasifinite.
    \item \label{item:I-annihilates-V} There exists an ideal $I$ of $A$ of finite codimension such that $(\g \otimes I)V=0$.
    \item \label{item:psi-zero-on-I} There exists an ideal $I$ of $A$ of finite codimension such that $\psi(\h \otimes I)=0$.
  \end{enumerate}
  If $A$ is finitely generated, then the above conditions are also equivalent to:
  \begin{enumerate}[resume]
    \item \label{item:V-finite-support} The module $V$ has finite support.
  \end{enumerate}
\end{theo}

\begin{proof}
  \eqref{item:V-quasifinite} $\Rightarrow$ \eqref{item:I-annihilates-V}: Let $\lambda$ be the highest weight of $V$.  Let $\alpha$ be a positive root of $\g$ and let $I_\alpha$ be the kernel of the linear map
  \[
    A \to \Hom_\kk (V_\lambda \otimes \g_{-\alpha},V_{\lambda-\alpha}),\quad f \mapsto (v \otimes u \mapsto (u \otimes f)v),\quad f \in A,\ v \in V_\lambda,\ u \in \g_{-\alpha}.
  \]
  It is clear that $I_\alpha$ is a linear subspace of $A$ of finite codimension.  We claim that $I_\alpha$ is in fact an ideal of $A$.  Indeed, since $\alpha \ne 0$, we can choose $h \in \h$ such that $\alpha(h) \ne 0$.  Then, for all $g \in A$, $f \in I_\alpha$, $v \in V_\lambda$, and $u \in \g_{-\alpha}$, we have
  \begin{align*}
    0 &= (h \otimes g)(u \otimes f)v \\
    &= [h \otimes g, u \otimes f]v + (u \otimes f)(h \otimes g)v \\
    &= -\alpha(h)(u \otimes gf)v + (u \otimes f)(h \otimes g)v.
  \end{align*}
  Now, since $(h \otimes g)v \in V_\lambda$ and $f \in I_\alpha$, the last term above is zero.  Since we also have $\alpha(h) \ne 0$, this implies that $(u \otimes gf)v=0$.  As this holds for all $v \in V_\lambda$ and $u \in \g_{-\alpha}$, we have $gf \in I_\alpha$.  Hence $I_\alpha$ is an ideal of $A$.

  Let $I$ be the intersection of all the $I_\alpha$.  Since $\g$ is finite dimensional (and thus has a finite number of positive roots), this intersection is finite and thus $I$ is also an ideal of $A$ of finite codimension.  We then have $(\n^- \otimes I)V_\lambda = 0$.  Since $\lambda$ is the highest weight of $V$, we also have $(\n^+ \otimes A)V_\lambda=0$.  Then, since $\h \otimes I \subseteq [\n^+ \otimes A, \n^- \otimes I]$,  we have $(\h \otimes I)V_\lambda=0$.  It follows that $(\g \otimes I)V_\lambda=0$.  Since $V_\lambda \ne 0$, it follows from Lemma~\ref{lem:annilating-ideal} that $(\g \otimes I)V=0$.

  \eqref{item:I-annihilates-V} $\Rightarrow$ \eqref{item:psi-zero-on-I}: Suppose \eqref{item:I-annihilates-V} holds and let $v$ be a highest weight vector of $V$.  Then, for any $\mu \in \h \otimes I$, $\psi(\mu)v = \mu v = 0$.  Thus $\psi(\h \otimes I)=0$.

  \eqref{item:psi-zero-on-I} $\Rightarrow$ \eqref{item:V-quasifinite}:  Suppose \eqref{item:psi-zero-on-I} holds.  Then, by Lemma~\ref{lem:psi-annilator}, $V(\psi)$ is naturally a module for the finite dimensional Lie superalgebra $\g \otimes (A/I)$.  Now, by the PBW Theorem (Lemma~\ref{lem:PBW}),
  \[
    V(\psi) = U(\g \otimes (A/I)) v = U(\n^- \otimes (A/I))v.
  \]
  Again by the PBW Theorem, the weight spaces of $U(\n^- \otimes (A/I))$ are finite dimensional.  Hence so are those of $V(\psi)$.

  Now suppose $A$ is finitely generated.  We prove that \eqref{item:I-annihilates-V} $\Leftrightarrow$ \eqref{item:V-finite-support}.  By definition, we have $\Supp_A V(\psi) = \Supp \Ann_A V(\psi)$.  Clearly \eqref{item:I-annihilates-V} is true if and only if $\Ann_A V(\psi)$ is of finite codimension.  Since $A$ is finitely generated, $\Ann_A V(\psi)$ is of finite codimension if and only if it has finite support (see Lemmas~\ref{lem:ideal-codim-implies-support} and~\ref{lem:ideal-support-implies-codim}).
\end{proof}

\begin{cor} \label{cor:finite-codim-annihilator}
  Let $V$ be a finite dimensional irreducible $(\g \otimes A)$-module, where $\g$ is either a basic classical Lie superalgebra or $\mathfrak{sl}(n,n)$, $n \ge 1$.  Then $(\g \otimes I) V = 0$ for some ideal $I$ of $A$ of finite codimension.
\end{cor}

\begin{proof}
  Since finite dimensional modules are clearly quasifinite, this follows immediately from Theorem~\ref{thm:hw-quasifinite-classification}.
\end{proof}

%
\section{Generalized evaluation modules} \label{sec:eval-rep}
%

In this section we assume that $\Gamma$ is a \emph{finite} group, acting on an affine scheme $X$ (see Remark~\ref{rem:reduction-to-affine}) and a \emph{finite dimensional} Lie superalgebra $\g$.  Furthermore, we assume that $A$ is finitely generated (hence Noetherian) and that $\Gamma$ acts \emph{freely} on $X$ (i.e.\ $\Gamma$ acts freely on $X_\rat = \maxSpec A$).

\begin{defin}[Generalized evaluation map] \label{def:gen-eval-map}
  Suppose $\sm_1,\dots,\sm_\ell \in X_\rat$ are pairwise distinct and $n_1,\dots,n_\ell \in \N_+$.  The associated \emph{generalized evaluation map} is the composition
  \begin{equation} \label{eq:gen-eval-map-def} \ts
    \ev_{\sm_1^{n_1},\dots,\sm_\ell^{n_\ell}} : \g \otimes A \twoheadrightarrow (\g \otimes A)/\left(\g \otimes \prod_{i=1}^\ell \sm_i^{n_i}\right) \cong \bigoplus_{i=1}^\ell \left( \g \otimes (A/\sm_i^{n_i}) \right).
  \end{equation}
  We let $\ev^\Gamma_{\sm_1^{n_1},\dots,\sm_\ell^{n_\ell}}$ denote the restriction of $\ev_{\sm_1^{n_1},\dots,\sm_\ell^{n_\ell}}$ to $(\g \otimes A)^\Gamma$.

  If $n_1=\dots=n_\ell=1$, then~\eqref{eq:gen-eval-map-def} is called an \emph{evaluation map}.  If $\sM = \{\sm_1,\dots,\sm_\ell\} \subseteq X_\rat$, we will sometimes write $\ev_\sM$ (resp.\ $\ev^\Gamma_\sM$) for $\ev_{\sm_1,\dots,\sm_\ell}$ (resp.\ $\ev^\Gamma_{\sm_1,\dots,\sm_\ell}$).
\end{defin}

\begin{defin}[Generalized evaluation representation] \label{def:gen-eval-rep}
  Suppose $\sm_1,\dots,\sm_\ell \in X_\rat$ are pairwise distinct, $n_1,\dots,n_\ell \in \N_+$, and, for $i=1,\dots,\ell$, $V_i$ is a finite dimensional $(\g \otimes (A/\sm^{n_i}))$-module with corresponding representation $\rho_i : \g \otimes (A/\sm^{n_i}) \to \End_\kk V_i$.  Then the composition
  \[ \ts
    \g \otimes A \xrightarrow{\ev_{\sm_1^{n_1},\dots,\sm_\ell^{n_\ell}}} \bigoplus_{i=1}^\ell \left( \g \otimes (A/\sm_i^{n_i}) \right) \xrightarrow{\bigotimes_{i=1}^\ell \rho_i} \End \left( \bigotimes_{i=1}^\ell V_i \right)
  \]
  is called a \emph{generalized evaluation representation} of $\g \otimes A$, and is denoted
  \[
    \ev_{\sm_1^{n_1},\dots,\sm_\ell^{n_\ell}} (\rho_1,\dots,\rho_\ell).
  \]
  The corresponding module is called a \emph{generalized evaluation module} and is denoted
  \[
    \ev_{\sm_1^{n_1},\dots,\sm_\ell^{n_\ell}} (V_1,\dots,V_\ell).
  \]
  We define $\ev^\Gamma_{\sm_1^{n_1},\dots,\sm_\ell^{n_\ell}} (\rho_1,\dots,\rho_\ell)$ to be the restriction of $\ev_{\sm_1^{n_1},\dots,\sm_\ell^{n_\ell}} (\rho_1,\dots,\rho_\ell)$ to $(\g \otimes A)^\Gamma$.  The notation $\ev^\Gamma_{\sm_1^{n_1},\dots,\sm_\ell^{n_\ell}} (V_1,\dots,V_\ell)$ is defined similarly (by restriction of the action).  These are also called \emph{(twisted) generalized evaluation representations} and \emph{(twisted) generalized evaluation modules} respectively.

  If $n_1=\dots=n_\ell=1$, then the above are called \emph{(twisted) evaluation representations} and \emph{(twisted) evaluation modules}, respectively.  If $\sM \subseteq X_\rat$ and $\rho_\sm : \g \to \End V_\sm$ is a finite dimensional representation for each $\sm \in \sM$, then $\ev_\sM (\rho_\sm)_{\sm \in \sM}$ and $\ev^\Gamma_\sM (\rho_\sm)_{\sm \in \sM}$ denote the obvious evaluation representations, where we have made the natural identification $\g \otimes (A/\sm) \cong \g \otimes \kk \cong \g$ for $\sm \in \sM$.
\end{defin}

\begin{rem}
  Note that in Definitions~\ref{def:gen-eval-map} and~\ref{def:gen-eval-rep} we allow for evaluation at more than one point.  In many places in the literature on multiloop algebras, the term \emph{evaluation representation} is reserved for the single point case.  One easily sees that the tensor product of generalized evaluation representations is again a generalized evaluation representation and that
  \[ \ts
    \ev_{\sm_1^{n_1},\dots,\sm_\ell^{n_\ell}}(\rho_1,\dots,\rho_\ell) = \bigotimes_{i=1}^\ell \ev_{\sm_i^{n_i}}(\rho_i).
  \]
  That is, generalized evaluation representations are tensor products of single point generalized evaluation representations.
\end{rem}

\begin{prop} \label{prop:gen-eval-finite-support}
  Suppose $\g$ is a reductive Lie algebra, a basic classical Lie superalgebra, or $\mathfrak{sl}(n,n)$, $n \ge 1$.  Then we have the following:
  \begin{enumerate}
    \item \label{prop-item:gen-eval-rep-iff-finite-support} An irreducible finite dimensional $(\g \otimes A)$-module is a generalized evaluation module if and only if it has finite support.
    \item \label{prop-item:eval-rep-iff-radical-annihilator} An irreducible finite dimensional $(\g \otimes A)$-module is an evaluation module if and only if it has finite reduced support.
  \end{enumerate}
\end{prop}

\begin{proof}
  We work in the equivalent language of representations (instead of modules).  Let $\rho$ be an irreducible representation of $\g \otimes A$.  If $\rho$ is a generalized evaluation representation then we have $\rho = \ev_{\sm_1^{n_1},\dots,\sm_\ell^{n_\ell}} (\rho_1,\dots,\rho_\ell)$ for some pairwise distinct $\sm_1,\dots,\sm_\ell \in X_\rat$, some $n_1,\dots,n_\ell \in \N_+$, and some representations $\rho_1,\dots,\rho_\ell$.  Then $I = \prod_{i=1}^\ell \sm_i^{n_i}$ has finite support and $\rho(\g \otimes I)=0$.  Thus $\rho$ has finite support (see Definition~\ref{def:support}).  Furthermore, if $\rho$ is an evaluation representation, then we can take $n_1=\dots=n_\ell=1$ above and $\prod_{i=1}^\ell \sm_i$ is a radical ideal.  Thus we have proved the forward implications in both~\eqref{prop-item:gen-eval-rep-iff-finite-support} and~\eqref{prop-item:eval-rep-iff-radical-annihilator}.

  Now assume $\rho(\g \otimes I)=0$ for some ideal $I$ of $A$ of finite support.  Thus $\rad I = \sm_1 \cdots \sm_\ell$ for some distinct maximal ideals $\sm_1,\dots,\sm_\ell$ of $A$.  Since $A$ is Noetherian, by Lemma~\ref{lem:ideal-contains-power-of-radical} we have that $\prod_{i=1}^\ell \sm_i^n \subseteq I$ for some $n \in \N$.  Then $\rho$ factors through the map
  \[ \ts
    \g \otimes A \twoheadrightarrow (\g \otimes A)/(\g \otimes \prod_{i=1}^\ell \sm_i^n) \cong \bigoplus_{i=1}^\ell (\g \otimes (A/\sm_i^n)).
  \]
  Then, by Corollary~\ref{cor:irreds-for-tensor-prods}, there exist representations $\rho_i$ of $\g \otimes (A/\sm_i^n)$, $i=1,\dots,\ell$, such that $\rho = \ev_{\sm_1^n,\dots,\sm_\ell^n}(\rho_1,\dots,\rho_\ell)$.  Thus $\rho$ is a generalized evaluation representation.  Furthermore, if $(\g \otimes I)V=0$ for some radical ideal, then we can take $n_1=\dots=n_\ell=1$ above and $\rho$ is an evaluation representation.  This completes the proof of the reverse implications in~\eqref{prop-item:gen-eval-rep-iff-finite-support} and~\eqref{prop-item:eval-rep-iff-radical-annihilator}.
\end{proof}

\begin{defin}[The set $X_*$]
  We let $X_*$ denote the set of finite subsets $\sM \subseteq X_\rat$ having the property that $\sm' \not \in \Gamma \sm$ for distinct $\sm,\sm' \in \sM$.
\end{defin}

\begin{lem} \label{lem:ev-image}
  For $\{\sm_1,\dots,\sm_\ell\} \in X_*$ and $n_1,\dots,n_\ell \in \N_+$, the map $\ev^\Gamma_{\sm_1^{n_1},\dots,\sm_\ell^{n_\ell}}$ is surjective.
\end{lem}

\begin{proof}
  Let $\{\sm_1,\dots,\sm_\ell\} \in X_*$ and $n_1,\dots,n_\ell \in \N_+$.  Then we have the natural surjection
  \[ \ts
    \g \otimes A \twoheadrightarrow \g \otimes A/\left( \prod_{\gamma \in \Gamma,\, 1 \le i \le \ell} \gamma \sm_i^{n_i} \right) \cong \bigoplus_{\gamma \in \Gamma} \bigoplus_{i=1}^\ell \g \otimes (A/(\gamma \sm_i)^{n_i}).
  \]
  Since the right hand side is $\Gamma$-invariant, we can restrict to $\Gamma$-fixed points to obtain a surjection
  \[ \ts
    (\g \otimes A)^\Gamma \twoheadrightarrow \left( \bigoplus_{\gamma \in \Gamma} \bigoplus_{i=1}^\ell \g \otimes (A/(\gamma \sm_i)^{n_i}) \right)^\Gamma \cong \bigoplus_{i=1}^\ell \g \otimes (A/\sm_i^{n_i}),
  \]
  where, in the isomorphism, we have used the fact that $\Gamma$ permutes the summands in the summation over $\Gamma$ in the middle expression.  Since the above composition is simply the map $\ev^\Gamma_{\sm_1^{n_1},\dots,\sm_\ell^{n_\ell}}$, the proof is complete.
\end{proof}

Note that Lemma~\ref{lem:ev-image} implies that, for any finite subset $\sM \subseteq X_\rat$, the map $\ev_\sM$ is surjective.

\begin{rem}
  In \cite[Prop~4.9]{NSS09} it was shown that if $\g$ is a Lie algebra, $\sM \in X_*$, and $\rho_\sm : \g \to \End V_\sm$ is an irreducible finite dimensional representation for each $\sm \in \sM$, then the evaluation representation $\ev^\Gamma_\sM(\rho_\sm)_{\sm \in \sM}$ is an irreducible finite dimensional representation of $(\g \otimes A)^\Gamma$.  However, this is no longer true in the setting of arbitrary Lie superalgebras.  This is because the (outer) tensor product of irreducible representations may be reducible in general (see \cite[\S8]{Che95} and \cite{Joz88}).  Thus, a possible alternate definition of (multiple point) evaluation representation in the super setting is a representation that factors through an evaluation map $\ev_\sM$.  However, Proposition~\ref{prop:tensor-product} implies that for the purposes of our classification (where we will assume that $\g$ is either a basic classical Lie superalgebra or $\mathfrak{sl}(n,n)$, $n \ge 1$), the two definitions are equivalent.
\end{rem}

We would like to give a natural enumeration of the isomorphism classes of evaluation representations of $(\g \otimes A)^\Gamma$.  Let $\cR(\g)$ denote the set of isomorphism classes of irreducible finite dimensional representations of $\g$. Then $\Gamma$ acts on $\cR(\g)$ by
\begin{equation} \label{eq:Gamma-action-on-R}
  \Gamma \times \cR(\g) \to \cR(\g),\quad (\gamma,[\rho]) \mapsto \gamma [\rho] := [\rho \circ \gamma^{-1}],
\end{equation}
where $[\rho] \in \cR$ denotes the isomorphism class of a representation $\rho$ of $\g$.

\begin{defin}[The sets $\cE(X,\g), \cE(X,\g)^\Gamma$]
  Let $\cE(X,\g)$ denote the set of finitely supported functions $\Psi : X_\rat \to \cR(\g)$ and let $\cE(X,\g)^\Gamma$ denote the subset of $\cE(X,\g)$ consisting of those functions that are $\Gamma$-equivariant.  Here the support $\Supp \Psi$ of $\Psi \in \cE(X,\g)$ is the set of all $\sm \in X_\rat$ for which $\Psi(\sm) \ne 0$, where $0$ denotes the isomorphism class of the trivial representation.
\end{defin}

For isomorphic representations $\rho$ and $\rho'$ of $\g$, the evaluation representations $\ev_\sm \rho$ and $\ev_\sm \rho'$ are isomorphic. Therefore, for $[\rho] \in \cR(\g)$, we can define $\ev_\sm [\rho]$ to be the isomorphism class of $\ev_\sm \rho$ and this is independent of the representative $\rho$. Similarly, for a finite
subset $\sM \subseteq X_\rat$ and representations $\rho_\sm$ of $\g$ for $\sm \in \sM$, we define $\ev_\sM ([\rho_\sm])_{\sm \in \sM}$ to be the isomorphism class of $\ev_\sM (\rho_\sm)_{\sm \in \sM}$.

\begin{lem} \label{lem:twisted-eval-invariance}
Suppose $\Psi \in \cE(X,\g)^\Gamma$ and $\sm \in X_\rat$.  Then, for all $\gamma \in \Gamma$,
\[
  \ev_\sm \Psi (\sm) = \ev_{\gamma  \sm} \left( \gamma \Psi(\sm) \right) = \ev_{\gamma  \sm} \Psi(\gamma  \sm).
\]
\end{lem}

\begin{proof}
  The proof is the same as that of \cite[Lem.~4.13]{NSS09}, which considers the case when $\g$ is a Lie algebra.
\end{proof}

\begin{defin}[The class $\ev^\Gamma_\Psi$]
  For $\Psi \in \cE(X,\g)^\Gamma$, we define $\ev^\Gamma_\Psi = \ev^\Gamma_\sM (\Psi(\sm))_{\sm \in \sM}$ where $\sM \in X_*$ contains one element of each $\Gamma$-orbit in $\Supp \Psi$.  By Lemma~\ref{lem:twisted-eval-invariance}, $\ev^\Gamma_\Psi$ is independent of the choice of $\sM$.  If $\Psi$ is the map that is identically 0 on $X$, we define $\ev^\Gamma_\Psi$ to be the isomorphism class of the trivial representation of $(\g \otimes A)^\Gamma$.  If $\Gamma$ is the trivial group, we often omit the superscript $\Gamma$.
\end{defin}

\begin{prop}
  The map $\Psi \mapsto \ev^\Gamma_\Psi$ from $\cE(X,\g)^\Gamma$ to the set of isomorphism classes of evaluation representations of $(\g \otimes A)^\Gamma$ is injective.
\end{prop}

\begin{proof}
  The proof is the same as that of \cite[Prop.~4.15]{NSS09}, which considers the case when $\g$ is a Lie algebra.
\end{proof}

\begin{prop} \label{prop:E-enumerates}
  If $\g$ is a reductive Lie algebra, a basic classical Lie superalgebra, or $\mathfrak{sl}(n,n)$, $n \ge 1$, then $\ev^\Gamma_\Psi$ is an irreducible representation for all $\Psi \in \cE(X,\g)^\Gamma$.  Hence the map $\Psi \mapsto \ev^\Gamma_\Psi$ is a bijection from $\cE(X,\g)^\Gamma$ to the set of isomorphism classes of irreducible evaluation representations of $(\g \otimes A)^\Gamma$.
\end{prop}

\begin{proof}
  This follows immediately from Corollary~\ref{cor:prod-fd-irr-is-irr}.
\end{proof}

We conclude this section by giving a natural enumeration of the generalized evaluation representations when $\g$ is a reductive Lie algebra.  Suppose $\fl^\ab$ is a finite dimensional abelian Lie algebra.

\begin{defin}[The spaces $\cL(X,\fl^\ab)$, $\cL_\sm(X,\fl^\ab)$]
  Let $\cL(X,\fl^\ab)$ be the space of linear forms on $\fl^\ab \otimes A$ with finite support.  Equivalently,
  \[
    \cL(X,\fl^\ab) := \{\theta \in (\fl^\ab \otimes A)^*\ |\ \theta(\fl^\ab \otimes I)=0 \text{ for some ideal $I$ of $A$ with finite support}\}.
  \]
  For $\sm \in \maxSpec A$, let
  \[
    \cL_\sm(X,\fl^\ab) := \{\theta \in (\fl^\ab \otimes A)^* \ |\ \theta(\fl^\ab \otimes \sm^\ell) = 0 \text{ for } \ell \gg 0\}
  \]
  be the linear subspace of $(\fl \otimes A)^*$ consisting of those forms vanishing on $\fl^\ab \otimes \sm^\ell$ for some sufficiently large $\ell$ (equivalently, forms with support contained in $\{\sm\}$).
\end{defin}

\begin{lem} \label{lem:cL-sum-decomp}
  We have $\bigoplus_{\sm \in X_\rat} \cL_\sm(X,\fl^\ab) \cong \cL(X,\fl^\ab)$.
\end{lem}

\begin{proof}
  We abbreviate $\cL_\sm = \cL_\sm(X,\fl^\ab)$ and $\cL = \cL(X,\fl^\ab)$.  Since $\cL_\sm \subseteq \cL$ for all $\sm \in X_\rat$, we have a natural map $\varphi : \bigoplus_{\sm \in X_\rat} \cL_\sm \to \cL$ given by $\varphi ((\theta_\sm)_{\sm \in X_\rat}) = \sum_{\sm \in X_\rat} \theta_\sm$.  Any element in the image of $\varphi$ can be written as a finite sum $\sum_{i=1}^n \theta_{\sm_i}$ for some pairwise distinct $\sm_i \in X_\rat$, $i=1,\dots,n$, where $\theta_{\sm_i} \in \cL_{\sm_i}$.  For each $i=1,\dots,n$, choose $\ell_i \in \N$ such that $\theta_{\sm_i}(\fl^\ab \otimes \sm_i^{\ell_i})=0$.  Let $I=\prod_{i=1}^n \sm_i^{\ell_i}$.  Then $\sum_{i=1}^n \theta_{\sm_i}$ factors as
  \[ \ts
    \fl^\ab \otimes A \twoheadrightarrow (\fl^\ab \otimes A)/(\fl^\ab \otimes I) \cong \bigoplus_{i=1}^n (\fl^\ab \otimes A/\sm_i^{\ell_i}) \to \kk,
  \]
  where $\theta_{\sm_i}$ is only nonzero on the $i$-th summand.  This shows that $\varphi$ is injective.

  Any element $\theta$ of $\cL$ factors through $\fl^\ab \otimes (A/I)$ for some ideal $I$ of $A$ with finite support.  By Lemma~\ref{lem:ideal-contains-power-of-radical}, $I \supseteq \prod_{i=1}^n \sm_i^\ell$ for some distinct maximal ideals $\sm_1,\dots,\sm_n$ and $\ell \in \N$.  Hence $\theta$ also factors through $\fl^\ab \otimes \left( A/\left( \prod_{i=1}^n \sm_i^\ell \right) \right) \cong \bigoplus_{i=1}^n (\fl^\ab \otimes A/\sm_i^\ell)$.  Let $\theta_i$, $i=1,\dots,n$ be the projection $\fl^\ab \otimes A \twoheadrightarrow \fl^\ab \otimes A/\sm_i^\ell$ followed by the restriction of $\theta$ to this summand.  Then $\theta = \sum_{i=1}^n \theta_i$.  Thus $\varphi$ is surjective.
\end{proof}

The following proposition classifies the irreducible finite dimensional generalized evaluation representations of $\fl \otimes A$, where $\fl$ is a reductive Lie algebra.  We let $\fl^\rss = [\fl,\fl]$ be its semisimple part and $\fl^\ab$ be its center, so that $\fl = \fl^\rss \oplus \fl^\ab$.

\begin{prop} \label{prop:classification-map-alg-modules}
  The map
  \[
    (\fl^\ab \otimes A)^* \times \cE(X,\fl^\rss) \to \cR(X,\fl),\quad (\theta,\Psi) \mapsto \theta \otimes \ev_\Psi,
  \]
  is a bijection, where $\cR(X,\fl)$ is the set of isomorphism classes of irreducible finite dimensional $(\fl \otimes A)$-modules.  In addition, $\theta \otimes \ev_\Psi$ is a generalized evaluation module if and only if $\theta \in \cL(X,\fl^\ab)$.  Thus the map
  \[
    (\theta, \Psi) \mapsto \theta \otimes \ev_\Psi,
  \]
  is a bijection from $\cL(X,\fl^\ab) \times \cE(X,\fl^\rss)$ to the set of isomorphism classes of irreducible finite dimensional generalized evaluation $(\fl \otimes A)$-modules.
\end{prop}

\begin{proof}
  It follows from Proposition~\ref{prop:density-theorem} and \cite[Lem.~2.7]{Li04} that all irreducible finite dimensional representations of $\fl \otimes A \cong (\fl^\ab \otimes A) \oplus (\fl^\rss \otimes A)$ are of the form $\theta \otimes \rho$, where $\theta \in (\fl^\ab \otimes A)^*$ is an irreducible finite dimensional (hence one-dimensional) representation of the abelian Lie algebra $\fl^\ab \otimes A$ and $\rho$ is an irreducible finite dimensional representation of $\fl^\rss \otimes A$.  By \cite[Thm.~5.5]{NSS09}, such $\rho$ are precisely the irreducible evaluation representations of $\fl^\rss \otimes A$.  This, together with Proposition~\ref{prop:E-enumerates}, proves the first bijection.

  For $\theta \in (\fl^\ab \otimes A)^*$ and $\Psi \in \cE(X,\fl^\rss)$, we have $\Ann_A (\theta \otimes \ev_\Psi) = (\Ann_A \theta) \cap (\Ann_A \ev_\Psi)$.  Thus
  \[
    \Supp_A (\theta \otimes \ev_\Psi) = (\Supp_A \theta) \cup (\Supp_A \ev_\Psi).
  \]
  Since $\Supp_A \ev_\Psi = \Supp \Psi$ is finite, we see that $\theta \otimes \ev_\Psi$ has finite support if and only if $\theta$ has finite support.  The second statement of the proposition then follows from Proposition~\ref{prop:gen-eval-finite-support}\eqref{prop-item:gen-eval-rep-iff-finite-support}.
\end{proof}

For $\sm \in X_\rat$, we have the natural projections
\[
  \cdots \twoheadrightarrow A/\sm^3 \twoheadrightarrow A/\sm^2 \twoheadrightarrow A/\sm \to 0.
\]
This gives rise to the sequence of injections
\[
  \cdots \hookleftarrow (A/\sm^3)^* \hookleftarrow (A/\sm^2)^* \hookleftarrow (A/\sm)^* \hookleftarrow 0.
\]
Via these injections, we view $(A/\sm^k)^*$ as a subspace of $(A/\sm^\ell)^*$ for $k \le \ell$.  The following lemma gives a concrete description of $\cL(X,\fl^\ab)$.

\begin{lem}
  For $\sm \in X_\rat$, we have $\cL_\sm(X,\fl^\ab) \cong \bigcup_{\ell=1}^\infty (\fl^\ab \otimes (A/\sm^\ell))^*$.  Therefore,
  \[
    \cL(X,\fl^\ab) \cong \bigoplus_{\sm \in X_\rat} \left( \bigcup_{\ell=1}^\infty (\fl^\ab \otimes (A/\sm^\ell))^* \right).
  \]
\end{lem}

\begin{proof}
  This follows immediately from the above discussion.
\end{proof}

\begin{rem}
  In this section we have assumed that $\Gamma$ acts freely on $X_\rat$ since it simplifies the definitions and we will make this assumption in our classification in Section~\ref{sec:classification-equivariant}.  To consider the case where $\Gamma$ does not act freely on $X_\rat$, we should modify the evaluation maps above to be maps to sums of subalgebras of $\g$ fixed by isotropy subgroups.  See \cite{NSS09} for these more general definitions in the case that $\g$ is a finite dimensional Lie algebra.
\end{rem}

%
\section{Kac modules and their irreducible quotients} \label{sec:kac-modules}
%

In this section we assume that $A$ is finitely generated and that $\g$ is a basic classical Lie superalgebra of type I or $\g = \mathfrak{sl}(n,n)$, $n \ge 1$.  Recall the distinguished $\Z$-grading $\g = \g_{-1} \oplus \g_0 \oplus \g_1$ of Section~\ref{subsec:Lie-superalg} and that $\g_{\bar 0} = \g_0$.

\begin{defin}[Modules $V(M)$] \label{def:V(M)}
  Let $M$ be an irreducible finite dimensional $(\g_0 \otimes A)$-module.  Let $\g_1 \otimes A$ act trivially on $M$ and define the induced module
  \[
    {\bar V}(M) = U(\g \otimes A) \otimes_{U((\g_0 \oplus \g_1) \otimes A)} M.
  \]
  The fact that a $(\g \otimes A)$-submodule of ${\bar V}(M)$ is proper if and only if it intersects $M$ nontrivially guarantees the existence of a unique maximal proper $(\g \otimes A)$-submodule $N(M)$ of ${\bar V}(M)$.  We define
  \[
    V(M) = {\bar V}(M)/N(M).
  \]
  It follows from the definition that $V(M)$ is an irreducible $(\g \otimes A)$-module and that $V(M_1) \cong V(M_2)$ as $(\g \otimes A)$-modules if and only if $M_1 \cong M_2$ as $(\g_0 \otimes A)$-modules.
\end{defin}

\begin{rem}
  Since $\g_0$ is a finite dimensional Lie algebra, the irreducible finite dimensional $(\g_0 \otimes A)$-modules were classified in \cite{NSS09} (see Proposition~\ref{prop:classification-map-alg-modules} of the current paper).  The module ${\bar V}(M)$ of Definition~\ref{def:V(M)} is a generalization of a \emph{Kac module}.  It is precisely a Kac module when $A \cong \kk$.  See \cite[Thm.~8]{Kac77}, \cite[p.~613]{Kac78}, or \cite[\S2.35]{FSS00}.
\end{rem}

\begin{lem} \label{lem:kac-module-annihilator}
  Suppose that $M$ is an irreducible finite dimensional $(\g_0 \otimes A)$-module and that $I$ is an ideal of $A$. Then $(\g_0 \otimes I)M=0$ if and only if $(\g \otimes I)V(M)=0$.  In particular,
  \begin{enumerate}
    \item $\Ann_A M = \Ann_A V(M)$,
    \item $\Supp_A M = \Supp_A V(M)$, and
    \item $V(M)$ is an evaluation module (resp.\ generalized evaluation module) if and only if $M$ is.
  \end{enumerate}
\end{lem}

\begin{proof}
  It suffices to prove the first statement since the others follow immediately from Definitions~\ref{def:annihilator} and~\ref{def:support}, and Proposition~\ref{prop:gen-eval-finite-support}.  It is clear that $(\g \otimes I)V(M)=0$ implies $(\g_0 \otimes I)M=0$ and so it remains to show the other implication.

  Let $I$ be an ideal of $A$ such that $(\g_0 \otimes I)M=0$.  It suffices to show that $(\g \otimes I) {\bar V}(M)$ is a proper submodule of ${\bar V}(M)$.  The fact that $(\g \otimes I){\bar V}(M)$ is a submodule follows easily from the fact that $\g \otimes I$ is an ideal of $\g \otimes A$.  It remains to show that this submodule is proper.

  Let $U_+(\g \otimes A) = \bigoplus_{i > 0} U_i(\g \otimes A)$, where $U_i(\g \otimes A)$ denotes the $i$-th step of the usual filtration on the enveloping superalgebra.  Then $U_+(\g \otimes A)$ is a subalgebra of $U(\g \otimes A)$ and we have a vector space decomposition
  \[
    {\bar V}(M) = (1 \otimes M) \oplus \left( U_+(\g_{-1} \otimes A) \otimes M \right).
  \]

  We first claim that $(\g_0 \otimes I) {\bar V}(M) \subseteq (\g_{-1} \otimes I) {\bar V}(M)$.  Note that $U_+(\g_{-1} \otimes A)$ is spanned by elements of the form
  \begin{equation} \label{eq:Ug-1A+element}
    (u_1 \otimes f_1)(u_2 \otimes f_2) \cdots (u_k \otimes f_k),\quad u_i \in \g_{-1},\ f_i \in A,\ i=1,\dots,k,\ k \in \N_+.
  \end{equation}
  For $u \in \g_0$ and $f \in I$, we have (recall that $\g_{-1}$ is abelian)
  \begin{multline*}
    (u \otimes f) (u_1 \otimes f_1)(u_2 \otimes f_2) \cdots (u_k \otimes f_k) = [u \otimes f, u_1 \otimes f_1](u_2 \otimes f_2) \cdots (u_k \otimes f_k) \\
    + (u_1 \otimes f_1)[u \otimes f, u_2 \otimes f_2] (u_3 \otimes f_3) \cdots (u_k \otimes f_k) + \cdots + (u_1 \otimes f_1) \cdots (u_{k-1} \otimes f_{k-1}) [u \otimes f, u_k \otimes f_k] \\
    + (u_1 \otimes f_1) \cdots (u_k \otimes f_k) (u \otimes f) \subseteq (\g_{-1} \otimes I)U(\g_{-1} \otimes A) + U(\g_{-1} \otimes A)(\g_0 \otimes I).
  \end{multline*}
  It therefore follows that
  \[
    (\g_0 \otimes I) U_+(\g_{-1} \otimes A) \subseteq (\g_{-1} \otimes I)U(\g_{-1} \otimes A) + U(\g_{-1} \otimes A)(\g_0 \otimes I).
  \]
  Thus,
  \begin{align*}
    (\g_0 \otimes I) {\bar V}(M) &= (\g_0 \otimes I)(1 \otimes M) + (\g_0 \otimes I)(U_+(\g_{-1} \otimes A) \otimes M) \\
    &\subseteq (\g_{-1} \otimes I)U(\g_{-1} \otimes A) \otimes M + U(\g_{-1} \otimes A)(\g_0 \otimes I) \otimes M \\
    &= (\g_{-1} \otimes I)U(\g_{-1} \otimes A) \otimes M \\
    &= (\g_{-1} \otimes I){\bar V}(M)
  \end{align*}
  as desired.

  Next we claim that $(\g_1 \otimes I) {\bar V}(M) \subseteq (\g_{-1} \otimes I) {\bar V}(M)$.  First, an argument similar to the one above shows that
  \[
    (\g_1 \otimes I)U(\g_{-1} \otimes A) \subseteq U(\g_{-1} \otimes A)(\g_0 \otimes I) U(\g_{-1} \otimes A) + U(\g_{-1} \otimes A) (\g_1 \otimes I).
  \]
  Then
  \begin{align*}
    (\g_1 \otimes I){\bar V}(M) &= (\g_1 \otimes I)U(\g_{-1} \otimes A) \otimes M \\
    &= U(\g_{-1} \otimes A)(\g_0 \otimes I) U(\g_{-1} \otimes A) \otimes M + U(\g_{-1} \otimes A)(\g_1 \otimes I) \otimes M \\
    &= U(\g_{-1} \otimes A)(\g_0 \otimes I) {\bar V}(M) \\
    &\subseteq U(\g_{-1} \otimes A)(\g_{-1} \otimes I){\bar V}(M) \\
    &= (\g_{-1} \otimes I)U(\g_{-1} \otimes A){\bar V}(M) \\
    &= (\g_{-1} \otimes I){\bar V}(M),
  \end{align*}
  proving our claim.

  We thus have
  \[
    (\g \otimes I) {\bar V}(M) = (\g_{-1} \otimes I){\bar V}(M) + (\g_0 \otimes I){\bar V}(M) + (\g_1 \otimes I){\bar V}(M) = (\g_{-1} \otimes I){\bar V}(M).
  \]
  Since ${\bar V}(M) \cong U(\g_{-1} \otimes A) \otimes M$ as vector spaces, it follows that $(\g_{-1} \otimes I){\bar V}(M)$ is a proper submodule of ${\bar V}(M)$.
\end{proof}

\begin{prop} \label{prop:fd-Kac-vs-gen-eval}
  The module $V(M)$ is finite dimensional if and only if $M$ is a generalized evaluation module.
\end{prop}

\begin{proof}
  Suppose $V(M)$ is finite dimensional.  By Corollary~\ref{cor:finite-codim-annihilator}, $(\g \otimes I)V(M)=0$ for some ideal $I$ of $A$ of finite codimension (and hence of finite support by Lemma~\ref{lem:ideal-codim-implies-support}).  Thus, by Proposition~\ref{prop:gen-eval-finite-support}\eqref{prop-item:gen-eval-rep-iff-finite-support}, $V(M)$ is a generalized evaluation module and, by Lemma~\ref{lem:kac-module-annihilator}, so is $M$.

  Now suppose that $M$ is a generalized evaluation module.  Then, by Proposition~\ref{prop:gen-eval-finite-support}\eqref{prop-item:gen-eval-rep-iff-finite-support}, there exists an ideal $I$ of $A$ of finite support such that $(\g \otimes I)M=0$.  Since $A$ is finitely generated, $I$ has finite codimension by Lemma~\ref{lem:ideal-support-implies-codim}.  By Lemma~\ref{lem:kac-module-annihilator}, we also have $(\g \otimes I)V(M)=0$.  It follows from Definition~\ref{def:V(M)} that $V(M)$ is a quotient of $(\g \otimes (A/I)) \otimes_{U((\g_0 \oplus \g_1) \otimes (A/I))} M$, which is isomorphic to $U(\g_{-1} \otimes (A/I)) \otimes_\kk M$ as a vector space.  Since $\g_{-1} \otimes (A/I)$ is odd and finite dimensional, $U(\g_{-1} \otimes (A/I))$ is finite dimensional by the PBW Theorem (Lemma~\ref{lem:PBW}).  Hence $V(M)$ is also finite dimensional.
\end{proof}

\begin{lem} \label{lem:V(M)-product}
  If $M_1,M_2$ are generalized evaluation $(\g_0 \otimes A)$-modules with disjoint supports, then $V(M_1 \otimes M_2) \cong V(M_1) \otimes V(M_2)$.
\end{lem}

\begin{proof}
  Suppose $M_1,M_2$ are generalized evaluation $(\g_0 \otimes A)$-modules with disjoint supports.  Then there exist ideals $I_1,I_2$ with disjoint supports such that $(\g_0 \otimes I_i)M_i=0$ for $i=1,2$.  Let $I=I_1I_2 = I_1 \cap I_2$ (see Lemma~\ref{lem:ideal-product-intersection}) and $M=M_1 \otimes M_2$.  Then $(\g_0 \otimes I)M = 0$ and thus, by Lemma~\ref{lem:kac-module-annihilator}, $(\g \otimes I)V(M)=0$.  Similarly, $(\g \otimes I_i)V(M_i)=0$ for $i=1,2$.  We can therefore consider $V(M)$, $V(M_1)$, and $V(M_2)$ as modules for $(\g \otimes (A/I))$, $(\g \otimes (A/I_1))$ and $(\g \otimes (A/I_2))$, respectively.  We have
  \begin{align*}
    {\bar V}(M) &= U(\g \otimes (A/I)) \otimes_{U((\g_0 \oplus \g_1) \otimes (A/I))} M \\
    &\cong U(\g \otimes (A/I_1 \oplus A/I_2)) \otimes_{U((\g_0 \oplus \g_1) \otimes (A/I_1 \oplus A/I_2))} (M_1 \otimes M_2) \\
    &\cong \left( U(\g \otimes (A/I_1)) \otimes U(\g \otimes (A/I_2)) \right) \otimes_{U((\g_0 \oplus \g_1) \otimes (A/I_1)) \otimes U((\g_0 \oplus \g_1) \otimes (A/I_2))} (M_1 \otimes M_2) \\
    &\cong \left(U(\g \otimes (A/I_1)) \otimes_{U((\g_0 \oplus \g_1) \otimes (A/I_1))} M_1 \right) \otimes \left(U(\g \otimes (A/I_2)) \otimes_{U((\g_0 \oplus \g_1) \otimes (A/I_2))} M_2 \right) \\
    &= {\bar V}(M_1) \otimes {\bar V}(M_2).
  \end{align*}
  Thus the maximal proper $(\g \otimes (A/I))$-submodule of ${\bar V}(M)$ is simply $(N(M_1) \otimes \bar V(M_2)) + (\bar V(M_1) \otimes N(M_2))$, where $N(M_i)$ is the maximal proper $(\g \otimes (A/I_i))$-submodule of ${\bar V}(M_i)$ for $i=1,2$.  Therefore $V(M) = ({\bar V}(M_1) \otimes {\bar V}(M_2))/\big( (N(M_1) \otimes \bar V(M_2)) + (\bar V(M_1) \otimes N(M_2)) \big) \cong ({\bar V}(M_1)/N(M_1)) \otimes ({\bar V}(M_2)/N(M_2)) = V(M_1) \otimes V(M_2)$ as desired.
\end{proof}

\begin{cor}
  For pairwise distinct $\sm_1,\dots,\sm_\ell \in X_\rat$, $n_1,\dots,n_\ell \in \N$, and irreducible finite dimensional $(\g_0 \otimes (A/\sm_i^{n_i}))$-modules $V_i$, $i=1,\dots,\ell$, we have
  \[ \ts
    V(\ev_{\sm_1^{n_1},\dots,\sm_\ell^{n_\ell}}(V_1,\dots,V_\ell)) \cong \bigotimes_{i=1}^\ell V(\ev_{\sm_i^{n_i}}(V_i)).
  \]
  Thus, if $M$ is a generalized evaluation $(\g_0 \otimes A)$-module, then $V(M)$ is a tensor product of modules of the form $V(M')$, where $M'$ is a single point generalized evaluation $(\g_0 \otimes A)$-module.  In particular, this is the case if $V(M)$ is finite dimensional.
\end{cor}

\begin{proof}
  Since $\ev_{\sm_1^{n_1},\dots,\sm_\ell^{n_\ell}}(V_1,\dots,V_\ell)) \cong \bigotimes_{i=1}^\ell \ev_{\sm_i^{n_i}}(V_i)$, the first result follows from Lemma~\ref{lem:V(M)-product} by induction.  The last statement follows from Proposition~\ref{prop:fd-Kac-vs-gen-eval}.
\end{proof}

\begin{rem}
  If $A=\kk$, Kac modules are also defined in the case that $\g$ is of type II, although the definition is slightly more complicated (see \cite[Thm.~8]{Kac77}, \cite[p.~613]{Kac78}, or \cite[\S2.35]{FSS00}).  However, for the purposes of our classification, we do not need the generalized notion of Kac modules in type II.
\end{rem}

%
\section{Classification of finite dimensional irreducible untwisted modules} \label{sec:classification-untwisted}
%

In this section we classify the irreducible finite dimensional modules for map superalgebras.  We assume that $\g$ is a basic classical Lie superalgebra or $\mathfrak{sl}(n,n)$, $n \ge 1$, and that $A$ is finitely generated.

\begin{theo} \label{thm:untwisted-classification-g0-semisimple}
  If $\g$ is a basic classical Lie superalgebra with $\g_{\bar 0}$ semisimple, then we have a bijection
  \begin{equation} \label{eq:classification-g0-semisimple}
    \cE(X,\g) \to \cR(X,\g),\quad \Psi \mapsto \ev_\Psi,
  \end{equation}
  where $\cR(X,\g)$ is the set of isomorphism classes of irreducible finite dimensional representations of $\g \otimes A$.  In particular, all irreducible finite dimensional representations are evaluation representations.
\end{theo}

\begin{proof}
  The map~\eqref{eq:classification-g0-semisimple} is injective by Proposition~\ref{prop:E-enumerates}.  To prove surjectivity, it suffices, by Lemma~\ref{lem:ideal-codim-implies-support} and Proposition~\ref{prop:gen-eval-finite-support}\eqref{prop-item:eval-rep-iff-radical-annihilator}, to show that for any irreducible finite dimensional $(\g \otimes A)$-module $V$, we have $(\g \otimes J)V=0$ for some radical ideal $J$ of $A$ of finite codimension.

  By Corollary~\ref{cor:finite-codim-annihilator}, $(\g \otimes I)V = 0$ for some ideal $I$ of $A$ of finite codimension.  Let $J = \rad I$.  We claim that $(\g \otimes J)V =0$.  By Lemma~\ref{lem:annilating-ideal}, it is enough to show that $\g \otimes J$ annihilates some nonzero vector of $V$.

  We can naturally consider $V$ as a $(\g \otimes (A/I))$-module.  It then suffices to show that $\g \otimes (J/I)$ annihilates some nonzero vector of $V$.  Since $A$ is Noetherian, some power of $J$ is contained in $I$ by Lemma~\ref{lem:ideal-contains-power-of-radical}.  Thus $\g \otimes (J/I)$ is solvable.  Since
  \begin{multline*}
    [(\g \otimes (J/I))_{\bar 1},(\g \otimes (J/I))_{\bar 1}] = [\g_{\bar 1},\g_{\bar 1}] \otimes (J^2/I) \\
    \subseteq \g_{\bar 0} \otimes (J^2/I) = [\g_{\bar 0},\g_{\bar 0}] \otimes (J^2/I) = [(\g \otimes (J/I))_{\bar 0}, (\g \otimes (J/I))_{\bar 0}],
  \end{multline*}
  it follows from Lemma~\ref{lem:fd-reps-solvable} that $V$ has a one-dimensional $(\g \otimes (J/I))$-invariant subspace.  Thus, there exists a nonzero vector $w \in V$ and $\theta \in (\g \otimes J)^*$ such that
  \[
    \mu w = \theta(\mu) w,\quad \forall\ \mu \in \g \otimes J.
  \]
  We claim that $\theta=0$.  For $\mu \in \n^{\pm} \otimes J$, we have $\theta(\mu)^m w = \mu^mw = 0$ for $m$ sufficiently large (since $V$ has a finite number of nonzero weight spaces).  Thus $\theta(\n^{\pm} \otimes J)=0$.  It remains to show that $\theta(\h \otimes J)=0$.  Denote the restriction of $\theta$ to $\g_{\bar 0} \otimes J$ by $\theta'$.  Then $\theta'$ defines a one-dimensional representation of the Lie algebra $\g_{\bar 0} \otimes J$.  Thus the kernel of $\theta'$ must be an ideal of $\g_{\bar 0} \otimes J$ of codimension at most one.  Since $\g_{\bar 0}$ is semisimple, it easily follows that this kernel must, in fact, be all of $\g_{\bar 0} \otimes J$.  Hence $\theta'=0$.  The fact that $\h \subseteq \g_{\bar 0}$ thus implies that $\theta(\h)=0$ as desired.
\end{proof}

Now suppose $\g$ is either a basic Lie superalgebra of type I or $\g = \fsl(n,n)$, $n \ge 1$.  Let $\g_{\bar 0}^\rss = [\g_{\bar 0}, \g_{\bar 0}]$ be the semisimple part of $\g_{\bar 0}$, and let $\g_{\bar 0}^\ab$ be its center.  Thus $\g_{\bar 0} = \g_{\bar 0}^\rss \oplus \g_{\bar 0}^\ab$.  Recall from Section~\ref{subsec:Lie-superalg} that we have the distinguished $\Z$-grading
\[
  \g = \g_{-1} \oplus \g_0 \oplus \g_1
\]
of $\g$ compatible with the $\Z_2$-grading and such that
\[
  \g_{\bar 0} = \g_0,\quad \g_{\bar 1} = \g_{-1} \oplus \g_1.
\]
We also have, using the choices of Section~\ref{subsec:Lie-superalg}, that $\g_1$ is a sum of positive (odd) root spaces, and $\g_{-1}$ is a sum of negative (odd) root spaces.

\begin{theo} \label{thm:untwisted-classification-type-I}
  If $\g$ is a basic Lie superalgebra of type I or $\g = \mathfrak{sl}(n,n)$, $n \ge 1$, then every irreducible finite dimensional $(\g \otimes A)$-module is of the form $V(M)$ for an irreducible generalized evaluation $(\g_{\bar 0} \otimes A)$-module $M$.  We thus have a natural bijection
  \begin{equation} \label{eq:g0-not-ss-bijection}
    \cL(X,\g_{\bar 0}^\ab) \times \cE(X,\g_{\bar 0}^\rss) \to \cR(X,\g),\quad (\theta, \Psi) \mapsto V(\theta \otimes \ev_\Psi),
  \end{equation}
  where $\cR(X,\g)$ is the set of isomorphism classes of irreducible finite dimensional $(\g \otimes A)$-modules.
\end{theo}

\begin{proof}
  Let $V$ be an irreducible finite dimensional $(\g \otimes A)$-module.  Thus $V \cong V(\psi)$ for some $\psi \in (\h \otimes A)^*$ by Lemma~\ref{lem:fd-implies-hw}.  Choose a highest weight vector $v$ and define
  \[
    M = U(\g_{\bar 0} \otimes A) v,
  \]
  a finite dimensional $(\g_{\bar 0} \otimes A)$-module.  Since $[\g_1 \otimes A, \g_{\bar 0} \otimes A] \subseteq \g_1 \otimes A$ and $(\g_1 \otimes A)v=0$, we have $(\g_1 \otimes A)M=0$.

  We claim that $M$ is irreducible as a $(\g_{\bar 0} \otimes A)$-module.  Let $\g_{\bar 0} = \n_{\bar 0}^+ \oplus \h \oplus \n_{\bar 0}^-$ be the triangular decomposition induced by the one on $\g$.  Then we have
  \[
    M = U(\n_{\bar 0}^- \otimes A) v.
  \]
  Let $w \in M$ be a weight vector.  Since $V(\psi)$ is irreducible as a $(\g \otimes A)$-module, by Proposition~\ref{prop:density-theorem} and the PBW Theorem (Lemma~\ref{lem:PBW}) there exists $X \in U(\n_{\bar 0}^+ \otimes A)U(\g_1 \otimes A)$ such that $Xw=v$.  Since $(\g_1 \otimes A)M=0$, we can in fact assume that $X \in U(\n_{\bar 0}^+ \otimes A)$.  Thus $M$ is irreducible as a $(\g_{\bar 0} \otimes A)$-module as desired.

  By Theorem~\ref{thm:hw-quasifinite-classification}, $(\g \otimes I)V=0$ for some ideal $I$ of $A$ of finite codimension (hence of finite support by Lemma~\ref{lem:ideal-codim-implies-support}).  This implies that $(\g_{\bar 0} \otimes I)M=0$ and hence, by Lemma~\ref{prop:gen-eval-finite-support}\eqref{prop-item:gen-eval-rep-iff-finite-support}, $M$ is a generalized evaluation $(\g_{\bar 0} \otimes A)$-module.

  It follows that $V$ is the irreducible quotient of $U(\g \otimes A) \otimes_{U((\g_0 \oplus \g_1) \otimes A)} M$, completing the proof of the first statement.

  By Lemma~\ref{lem:kac-module-annihilator}, every module of the form $V(M)$, for $M$ an irreducible generalized evaluation $(\g_{\bar 0} \otimes A)$-module, is an irreducible finite dimensional $(\g \otimes A)$-module.  The fact that the map~\eqref{eq:g0-not-ss-bijection} is a bijection then follows from Proposition~\ref{prop:classification-map-alg-modules}.
\end{proof}

\begin{rem} \label{rem:thm-overlap}
  Note that, taken together, Theorems~\ref{thm:untwisted-classification-g0-semisimple} and~\ref{thm:untwisted-classification-type-I} cover all basic classical Lie superalgebras $\g$.  Also, \emph{both} theorems apply when $\g = A(n,n)$, $n \ge 1$.  In this case we get two descriptions of the irreducible finite dimensional modules.
\end{rem}

%
\section{Classification of finite dimensional irreducible twisted modules} \label{sec:classification-equivariant}
%

We assume in this section that $\g$ is a basic classical Lie superalgebra (except in Proposition~\ref{prop:equivariant-ideal-form}) and that $A$ is finitely generated.  Let $\Gamma$ be a finite \emph{abelian} group acting on $\g$ and $X$ by automorphisms.  We assume that the action of $\Gamma$ on $X$ is free (i.e.\ the action on $X_\rat = \maxSpec A$ is free).  Let $\Xi$ be the character group of $\Gamma$. This is an abelian group, whose group operation we will write additively. Hence, $0$ is the character of the trivial one-dimensional representation, and if an irreducible representation affords the character $\xi$, then $-\xi$ is the character of the dual representation.

If $\Gamma$ acts on a superalgebra $B$ by automorphisms, it is well known that $B=\bigoplus_{\xi \in \Xi} B_\xi$ is a $\Xi$-grading, where $B_\xi$ is the isotypic component of type $\xi$. It follows that $(\g \otimes A)^\Gamma$ can be written as
\begin{equation} \label{eq:EMSA-grading}
  (\g \otimes A)^\Gamma = \ts \bigoplus_{\xi \in \Xi} \, \g_\xi \otimes A_{-\xi},
\end{equation}
since $\g = \bigoplus_\xi \g_\xi$ and $A=\bigoplus_\xi A_\xi$ are $\Xi$-graded and $(\g_\xi \otimes A_{\xi'})^\Gamma = 0$ if $\xi'\ne -\xi$. The decomposition \eqref{eq:EMSA-grading} is a superalgebra $\Xi$-grading.

\begin{prop} \label{prop:equivariant-ideal-form}
  Suppose $\g$ is a finite dimensional simple Lie superalgebra.  Then all ideals of $(\g \otimes A)^\Gamma$ are of the form $(\g \otimes I)^\Gamma = \bigoplus_{\xi \in \Xi} \g_\xi \otimes I_{-\xi}$, where $I = \bigoplus_{\xi \in \Xi} I_\xi$ is a $\Gamma$-invariant ideal of $A$.
\end{prop}

\begin{proof}
  Let $\mathfrak{I}$ be an ideal of $(\g \otimes A)^\Gamma$ and let
  \[
    z = \sum_{i=1}^s u_i \otimes f_i,\quad 0 \ne u_i \in \g_{\xi_i},\ f_i \in A_{-\xi_i},\ \xi_i \in \Xi,\ 1 \le i \le s,
  \]
  be an arbitrary element of $\mathfrak{I}$.  Without loss of generality, we may assume the set $\{u_1,\dots,u_s\}$ is linearly independent.

  \medskip

  \emph{Claim:} For $j=\{1,\dots,s\}$, we have $\g_{\xi_j} \otimes f_j \subseteq \mathfrak{I}$.

  \medskip

  This implies that $\mathfrak{I} = \bigoplus_{\xi \in \Xi'} \g_\xi \otimes I_{-\xi}$ for some subspaces $I_{-\xi} \subseteq A_{-\xi}$, $\xi \in \Xi'$, where
  \begin{equation} \label{eq:Xi'}
    \Xi' = \{\xi \in \Xi \ |\ \g_\xi \ne 0\}.
  \end{equation}

  \medskip

  \emph{Proof of claim:} Fix $j \in \{1,\dots,s\}$ and $u \in \g_{\xi_j}$.   Since $\g$ is an irreducible finite dimensional module over itself, Proposition~\ref{prop:density-theorem} (with $A=\kk$) implies that there exists $p \in U(\g)$ such that $p u_j = u$ and $p u_i = 0$ for all $i \ne j$.  Now, the $\Xi$-grading on $\g$ induces a $\Xi$-grading on $U(\g)$.  Fixing a basis for $\g$ compatible with its $\Xi$-grading, we can assume that $p$ is a sum of monomials in this basis of degree zero.  That is, we can write $p = \sum_{k=1}^r p_k$, where, for $k = 1,\dots,r$, we have
  \[
    p_k = \prod_{\ell=1}^{n_k} x_{k,\ell},\quad x_{k,\ell} \in \g_{\zeta_{k,\ell}},\quad \sum_{\ell=1}^{n_k} \zeta_{k,\ell} = 0.
  \]
  Now, by \cite[Lem.~4.4]{NS11}, we have $\prod_{\ell=1}^{n_k} A_{-\zeta_{k,\ell}} = A_0$.  Thus, there exist $g_{k,\ell,q} \in A_{-\zeta_{k,\ell}}$, $1 \le \ell \le n_k$, $1 \le q \le b$, $b \in \N$, with
  \[
    \sum_{q=1}^b \left( \prod_{\ell=1}^{n_k} g_{k,\ell,q} \right) = 1.
  \]
  Define
  \begin{equation} \label{eq:tilde-p-def}
    \tilde p_k = \sum_{q=1}^b \left( \prod_{\ell=1}^{n_k} (x_{k,\ell} \otimes g_{k,\ell,q}) \right),\quad \tilde p = \sum_{k=1}^r \tilde p_k,
  \end{equation}
  and note that $\tilde p \in U \left( (\g \otimes A)^\Gamma \right)$.  Then
  \begin{align*}
    \tilde p_k (u_i \otimes f_i) &= \sum_{q=1}^b \left( \prod_{\ell=1}^{n_k} (x_{k,\ell} \otimes g_{k,\ell,q}) \right) \left( u_i \otimes f_i \right) \\
    &= \sum_{q=1}^b \left( (p_k u_i) \otimes \left( \prod_{\ell=1}^{n_k} g_{k,\ell,q} \right) f_i \right) \\
    &= (p_k u_i) \otimes \left( \sum_{q=1}^b \left( \prod_{\ell=1}^{n_k} g_{k,\ell,q} \right) f_i \right) \\
    &= (p_k u_i) \otimes f_i.
  \end{align*}
  Thus
  \[
    \tilde p z = \left( \sum_{k=1}^r \tilde p_k \right) \left( \sum_{i=1}^s u_i \otimes f_i \right) = \sum_{i=1}^s \left( \sum_{k=1}^r p_k u_i \right) \otimes f_i = \sum_{i=1}^s (p u_i) \otimes f_i = u \otimes f_j.
  \]
  This completes the proof of the claim.

  \medskip

  \emph{Claim:} We have $A_{-\zeta} I_{-\xi} \subseteq I_{-\xi-\zeta}$ for all $\xi \in \Xi'$ and $\zeta \in \Xi$ such that $\xi + \zeta \in \Xi'$.

  \medskip

  \emph{Proof of claim:} Fix such $\zeta,\xi$ and $f \in A_{-\zeta}$.  Fix also nonzero elements $u \in \g_\xi$ and $v \in \g_{\xi+\zeta}$.  As in the proof of the previous claim, Proposition~\ref{prop:density-theorem} implies that there exists $p \in U(\g)$ such that $p u = v$.  As above, we can write
  \[
    p = \sum_{k=1}^r p_k,\quad p_k = \prod_{\ell=1}^{n_k} x_{k,\ell},\quad x_{k,\ell} \in \g_{\zeta_{k,\ell}},\quad \sum_{\ell=1}^{n_k} \zeta_{k,\ell} = \zeta.
  \]
  By \cite[Lem.~4.4]{NS11}, $\prod_{\ell=1}^{n_k} A_{-\zeta_{k,\ell}} = A_{-\zeta}$.  Thus, there exist $g_{k,\ell,q} \in A_{-\zeta_{k,l}}$, $1 \le \ell \le n_k$, $1 \le q \le b$, $b \in \N$, with
  \[
    \sum_{q=1}^b \left( \prod_{\ell=1}^{n_k} g_{k,\ell,q} \right) = f.
  \]
  Define $\tilde p$ as in~\eqref{eq:tilde-p-def}.  Then, as above, we have
  \[
    \tilde p (u \otimes f') = v \otimes ff' \quad \forall\ f' \in I_{-\xi}.
  \]
  This implies that $ff' \in I_{-\xi-\gamma}$ for all $f' \in I_{-\xi}$, completing the proof of the claim.

  \medskip

  For $\xi \in \Xi \setminus \Xi'$, define
  \[
    I_{-\xi} = \sum_{\zeta \in \Xi'} A_{-\xi+\zeta} I_{-\zeta}.
  \]
  Then we clearly have
  \[
    \mathfrak{I} = \bigoplus_{\xi \in \Xi} \g_\xi \otimes I_{-\xi}.
  \]
  It follows from the above that $I := \bigoplus_{\xi \in \Xi} I_{-\xi}$ is an ideal of $A$.  In fact, it is the ideal of $A$ generated by $\bigoplus_{\xi \in \Xi'} I_{-\xi}$.  Furthermore, $I$ is $\Gamma$-invariant since it is a $\Xi$-graded subspace of $A$.
\end{proof}

\begin{rem}
  In the case that $\g$ is a finite dimensional simple Lie algebra and $A = \kk[t_1^{\pm 1},\dots,t_n^{\pm 1}]$ is a Laurent polynomial ring in several variables, Proposition~\ref{prop:equivariant-ideal-form} was proven in \cite[Prop.~2.7]{Lau10}.
\end{rem}

\begin{lem} \label{lem:finite-support-codim}
  If $I$ is a $\Gamma$-invariant ideal of $A$, then $|\Supp I|$ is divisible by $|\Gamma|$ and
  \[
    \dim A_\xi/I_\xi \ge \frac{|\Supp I|}{|\Gamma|} \quad \forall\ \xi \in \Xi.
  \]
  In particular, if $\dim A_\xi/I_\xi$ is finite for any $\xi \in \Xi$, then $\Supp I$ is finite.
\end{lem}

\begin{proof}
  Let $I$ be a $\Gamma$-invariant ideal of $A$.  That $|\Supp I|$ is divisible by $|\Gamma|$ follows from the fact that $\Supp I$ is a $\Gamma$-invariant subset of $X_\rat$, on which $\Gamma$ acts freely.  For the inequality, it suffices to show that, for $\xi \in \Xi$, $\dim A_\xi/I_\xi \ge n$ for all $n \in \N$ satisfying $n \ge |\Supp I|/|\Gamma|$.  Fix such an $n \in \N$ and choose points $\sm_1,\dots,\sm_n \in \Supp I$ in distinct orbits.  For $i=1,\dots,n$, we can choose $f'_i \in A$ such that
  \[
    f'_i(\sm_j)=\delta_{ij},\quad \text{and} \quad f'_i(\gamma \sm_j)=0\quad \forall\ 1 \ne \gamma \in \Gamma,\ j=1,\dots,n.
  \]
  (Here and in what follows, the evaluation $f(\sm)$ of $f \in A$ at $\sm \in X_\rat$ is the element of $\kk$ corresponding to $f + \sm$ in the quotient $A/\sm \cong \kk$.)
  Then define
  \[
    f_i = \sum_{\gamma \in \Gamma} \xi(\gamma^{-1}) \gamma f'_i.
  \]
  One easily checks that $f_i \in A_\xi$ and $f_i(\Gamma \sm_j) = \delta_{ij}$.  Thus the set $\{f_1,\dots,f_n\}$ is linearly independent and spans a subspace of $A_\xi$ intersecting $I_\xi$ trivially.  This proves $\dim A_\xi/I_\xi \ge n$ as desired.
\end{proof}

\begin{lem} \label{lem:fd-equiv-rep-finite-support}
  Every finite dimensional $(\g \otimes A)^\Gamma$-module has finite support.
\end{lem}

\begin{proof}
  Let $V$ be a finite dimensional $(\g \otimes A)^\Gamma$-module.  By Proposition~\ref{prop:equivariant-ideal-form}, the annihilator of $V$ is of the form $\bigoplus_{\xi \in \Xi} \g_\xi \otimes I_{-\xi}$ for some $\Gamma$-invariant ideal $I = \bigoplus_{\xi \in \Xi} I_{-\xi}$ of $A$.  The action of $(\g \otimes A)^\Gamma = \bigoplus_{\xi \in \Xi} \g_\xi \otimes A_{-\xi}$ factors through the quotient
  \[ \ts
    \bigoplus_{\xi \in \Xi} \g_\xi \otimes A_{-\xi} \twoheadrightarrow \left( \bigoplus_{\xi \in \Xi} \g_\xi \otimes A_{-\xi} \right) / \left( \bigoplus_{\xi \in \Xi} \g_\xi \otimes I_{-\xi} \right) \cong \bigoplus_{\xi \in \Xi} \g_\xi \otimes (A_{-\xi}/I_{-\xi}) \hookrightarrow \End V.
  \]
  Since $\End V$ is finite dimensional, we have that $A_{-\xi}/I_{-\xi}$ is finite dimensional for all $\xi \in \Xi'$, where $\Xi'$ is defined as in~\eqref{eq:Xi'}.  Since $\g \ne 0$, the set $\Xi'$ is nonempty.  Thus, by Lemma~\ref{lem:finite-support-codim}, $I$ has finite support and thus so does $V$.
\end{proof}

\begin{prop} \label{prop:fd-module-restriction}
  Every finite dimensional $(\g \otimes A)^\Gamma$-module $V$ is the restriction of a $(\g \otimes A)$-module $\tilde V$.  Furthermore, $V$ is irreducible if and only if $\tilde V$ is.
\end{prop}

\begin{proof}
  Let $V$ be a finite dimensional $(\g \otimes A)^\Gamma$-module and let $\rho : (\g \otimes A)^\Gamma \to \End V$ denote the corresponding representation.  By Proposition~\ref{prop:equivariant-ideal-form} and Lemma~\ref{lem:fd-equiv-rep-finite-support}, the kernel of $\rho$ is of the form $(\g \otimes I)^\Gamma$ for some $\Gamma$-invariant ideal $I$ of $A$ with finite support.  Since $A$ is finitely generated, Lemma~\ref{lem:ideal-support-implies-codim} implies that $I$ is of finite codimension in $A$.  The support of $I$ is a $\Gamma$-invariant subset of $X_\rat$.  Let $\sM \subseteq X_*$ contain one point from each $\Gamma$-orbit in the support of $I$.  Then
  \[ \ts
    I = \prod_{\sm \in \sM,\, \gamma \in \Gamma} \gamma I_\sm,
  \]
  where $I_\sm$ is an ideal with support $\{\sm\}$ for each $\sm \in \sM$,.  Thus
  \begin{align*}
    (\g \otimes A)^\Gamma/(\g \otimes I)^\Gamma &\cong (\g \otimes A/I)^\Gamma \\
    &\cong \ts \left( \g \otimes \bigoplus_{\sm \in \sM,\, \gamma \in \Gamma} A/(\gamma I_\sm) \right)^\Gamma \\
    &\cong \ts \left( \bigoplus_{\sm \in \sM,\, \gamma \in \Gamma} (\g \otimes A/(\gamma I_\sm)) \right)^\Gamma \\
    &\cong \ts \bigoplus_{\sm \in \sM} (\g \otimes A/I_\sm) \\
    &\cong \ts (\g \otimes A)/(\g \otimes J), \quad J = \prod_{\sm \in \sM} I_\sm,
  \end{align*}
  where, in the second-to-last isomorphism, we use the fact that, for $\sm \in \sM$, the group $\Gamma$ permutes the summands $\g \otimes A/(\gamma I_\sm)$, $\gamma \in \Gamma$.

  We now have the following commutative diagram, where $\tau$ is the above isomorphism, $\pi$ is the natural projection, and $\bar \rho$ is the map induced by $\rho$.
  \[
    \xymatrix{
      \g \otimes A \ar@{->>}[r]^(.35)\pi & (\g \otimes A)/(\g \otimes J) & \\
      (\g \otimes A)^\Gamma \ar@{^{(}->}[u] \ar@{->>}[r] & (\g \otimes A)^\Gamma/(\g \otimes I)^\Gamma \ar[u]^\tau_\cong \ar[r]^(0.65){\bar \rho} & \End V
    }
  \]
  It is clear that ${\bar \rho} \circ \tau^{-1} \circ \pi$ is a representation of $\g \otimes A$ that, when restricted to $(\g \otimes A)^\Gamma$, coincides with $\rho$.  Since both representations factor through the quotient $(\g \otimes A)^\Gamma/(\g \otimes I)^\Gamma$, one is irreducible if and only if the other is.
\end{proof}

\begin{lem} \label{lem:move-eval-rep-around-orbit}
  Suppose $\sm \in X_\rat$, $\ell \in \N$, and $\rho$ is a representation of $\g \otimes (A/\sm^\ell)$.  Then $\rho \circ \gamma^{-1}$ is a representation of $\g \otimes (A/\gamma \sm^\ell)$ and $\ev_{\gamma \sm^\ell} (\rho \circ \gamma^{-1}) \cong \ev_{\sm^\ell} (\rho) \circ \gamma^{-1}$.  Thus $\ev^\Gamma_{\gamma \sm^\ell} (\rho \circ \gamma^{-1}) \cong \ev^\Gamma_{\sm^\ell} (\rho)$.
\end{lem}

\begin{proof}
  We have
  \[
    \ev_{\gamma \sm^\ell}(\rho \circ \gamma^{-1}) = \rho \circ \gamma^{-1} \circ \ev_{\gamma \sm^\ell} = \rho \circ \gamma^{-1} \circ \big( \gamma \circ \ev_{\sm^\ell} \circ \gamma^{-1} \big) = \ev_{\sm^\ell}(\rho) \circ \gamma^{-1}.
  \]
  The second isomorphism follows from the fact that $\gamma^{-1}$ fixes $(\g \otimes A)^\Gamma$.
\end{proof}

\begin{cor} \label{cor:irr-fd-modules-restrictions}
  Every irreducible finite dimensional representation of $(\g \otimes A)^\Gamma$ is a generalized evaluation representation
  \[
    \ev^\Gamma_{\sm_1^{\ell_1},\dots,\sm_n^{\ell_n}} (\rho_1,\dots,\rho_m),
  \]
  such that
  \begin{enumerate}
    \item the maximal ideals $\sm_1,\dots,\sm_\ell$ are distinct,
    \item the representations $\rho_i$ are irreducible, and
    \item \label{cor-item:distinct-orbits} $\{\sm_1,\dots,\sm_n\} \in X_*$ (in other words, the $\sm_i$ all lie in distinct $\Gamma$-orbits).
  \end{enumerate}
\end{cor}

\begin{proof}
  By Proposition~\ref{prop:fd-module-restriction} and Theorems~\ref{thm:untwisted-classification-g0-semisimple} and~\ref{thm:untwisted-classification-type-I}, every irreducible finite dimensional representation of $(\g \otimes A)^\Gamma$ is a generalized evaluation representation $\ev^\Gamma_{\sm_1^{\ell_1},\dots,\sm_n^{\ell_n}} (\rho_1,\dots,\rho_m)$.  By Lemma~\ref{lem:move-eval-rep-around-orbit}, we can always assume that condition~\eqref{cor-item:distinct-orbits} is satisfied.  Then, by replacing the $\rho_i$ by appropriate tensor products, we can always assume the maximal ideals $\sm_i$ are distinct.  Now, if $\rho_i$ is reducible for some $i=1,\dots,n$, then $\ev^\Gamma_{\sm_1^{\ell_1},\dots,\sm_n^{\ell_n}} (\rho_1,\dots,\rho_m)$ is clearly also reducible, which is a contradiction.
\end{proof}

\begin{defin}[$\cL(X,\g_{\bar 0}^\ab)^\Gamma$]
  If $\theta \in \cL_\sm(X,\g_{\bar 0}^\ab)$ and $\gamma \in \Gamma$, then $\theta \circ \gamma^{-1} \in \cL_{\gamma \sm}(X,\g_{\bar 0}^\ab)$.  Thus $\Gamma$ acts on $\cL(X,\g_{\bar 0}^\ab)$ via
  \[
    \gamma \theta = \theta \circ \gamma^{-1},\quad \gamma \in \Gamma,\ \theta \in \cL(X,\g_{\bar 0}^\ab).
  \]
  Let $\cL(X,\g_{\bar 0}^\ab)^\Gamma$ be the subset of $\cL(X,\g_{\bar 0}^\ab)$ consisting of $\Gamma$-fixed points.
\end{defin}

\begin{defin}[$V^\Gamma(\theta,\Psi)$]
  Suppose $\g$ is of type I, $\theta \in \cL(X,\g_{\bar 0}^\ab)^\Gamma$, and $\Psi \in \cE(X,\g_{\bar 0}^\rss)^\Gamma$.  Choose $\sM \in X_*$ to contain one point in each $\Gamma$-orbit of $(\Supp_A \theta) \cup (\Supp \Psi)$.  Let $\theta_\sM$ be the image of $\theta$ under the composition
  \[ \ts
    \cL(X,\g_{\bar 0}^\ab)^\Gamma \hookrightarrow \cL(X,\g_{\bar 0}^\ab) = \bigoplus_{\sm \in X_\rat} \cL_\sm(X,\g_{\bar 0}^\ab) \twoheadrightarrow \bigoplus_{\sm \in \sM} \cL_\sm(X,\g_{\bar 0}^\ab)
  \]
  and define $\Psi_\sM$ by
  \[
    \Psi_\sM(\sm) :=
    \begin{cases}
      \Psi(\sm), & \sm \in \sM, \\
      0, & \sm \not \in \sM.
    \end{cases}
  \]
  Then we define $V^\Gamma(\theta,\Psi)$ to be $(\g \otimes A)^\Gamma$-module obtained by restriction from the $(\g \otimes A)$-module $V(\theta_\sM \otimes \ev_{\Psi_\sM})$.  By Lemma~\ref{lem:move-eval-rep-around-orbit}, this definition is independent of the choice of $\sM$.
\end{defin}

Recall that $\cR(\g_{\bar 0}^\rss)$ denotes the set of isomorphism classes of irreducible finite dimensional representations of $\g_{\bar 0}^\rss$ and that $\Gamma$ acts on $\cR(\g_{\bar 0}^\rss)$ as in~\eqref{eq:Gamma-action-on-R}.

\begin{theo} \label{thm:classification-equivariant}
  Suppose $\g$ is a basic classical Lie superalgebra and let $\cR(X,\g)^\Gamma$ be the set of isomorphism classes of irreducible finite dimensional representations of $(\g \otimes A)^\Gamma$.  Then we have the following:
  \begin{enumerate}
    \item If $\g_{\bar 0}$ is semisimple (i.e.\ $\g$ is of type II or is $A(n,n)$, $n \ge 1$), then the map
        \begin{equation} \label{eq:enumeration-g0-semsimple}
          \cE(X,\g)^\Gamma \to \cR(X,\g)^\Gamma,\quad \Psi \mapsto \ev^\Gamma_\Psi,
        \end{equation}
        is a bijection.

    \item If $\g _{\bar 0}$ is of type I, then the map
        \begin{equation} \label{eq:enumeration-g0-not-semsimple}
          \cL(X,\g_{\bar 0}^\ab)^\Gamma \times \cE(X,\g_{\bar 0}^\rss)^\Gamma \to \cR(X,\g)^\Gamma,\quad (\theta,\Psi) \mapsto V^\Gamma(\theta, \Psi),
        \end{equation}
        is a bijection.
  \end{enumerate}
\end{theo}

\begin{proof}
  The maps~\eqref{eq:enumeration-g0-semsimple} and~\eqref{eq:enumeration-g0-not-semsimple} are surjective by Corollary~\ref{cor:irr-fd-modules-restrictions} and Theorems~\ref{thm:untwisted-classification-g0-semisimple} and~\ref{thm:untwisted-classification-type-I}.  The proof of injectivity is almost identical to the proof of \cite[Prop.~4.15]{NSS09} and is therefore omitted.
\end{proof}

\begin{rem}
  In the case that $\g$ is a semisimple Lie algebra, twisting and untwisting functors were defined in \cite{FKKS12}.  These are isomorphisms between certain categories of representations of equivariant map algebras and their untwisted analogues.  These functors could also be defined in the super setting of the current paper.  We do not pursue this line of inquiry here.
\end{rem}

\begin{rem}
  A special case of Theorem~\ref{thm:classification-equivariant} gives a classification of irreducible finite dimensional representations of the multiloop superalgebras (see Example~\ref{eg:multiloop}).  In the untwisted case, this recovers the classification given in \cite{RZ04,Rao11}.  In the twisted case, the classification seems to be new.
\end{rem}


\bibliographystyle{alpha}
\bibliography{savage-map-superalgebra-biblist}

\end{document}